\documentclass[]{article}
\usepackage[margin=3cm]{geometry}
   \usepackage{tikz}
\usepackage[utf8]{inputenc}
\usepackage{mathtools,amsfonts,amsthm,mathrsfs,amssymb,bm}
\usepackage{textcomp}
\usepackage{bookmark}
\usepackage{microtype}
\usepackage{enumitem}
\usepackage{ulem}
\usepackage{todonotes}
\usepackage{algorithm}
\usepackage{algpseudocode}
\usepackage{float}

\newcommand{\cB}{\mathcal{B}}
\newcommand{\cD}{\mathcal{D}}
\newcommand{\cE}{\mathcal{E}}

\newcommand{\bD}{\mathbf{D}}
\newcommand{\wD}{\widetilde{\bD}}
\newcommand{\wN}{\widetilde{N}}
\newcommand{\wn}{\widetilde{n}}
\newcommand{\bH}{\mathbf{H}}

\newcommand{\bphi}{\bm{\phi}}

\newcommand{\pr}[1]{\mathbb{P}\left[#1\right]}
\newcommand{\esp}[1]{\mathbb{E}\left[#1\right]}

\newcommand{\pth}[1]{\left(#1\right )}

\DeclareMathOperator{\et}{and}

\usepackage[english]{babel}
\usepackage{lineno}
\usepackage{graphicx,multicol}
\usepackage{epic,eepic,epsfig}
\usepackage{amssymb}
\usepackage{color}
\usepackage{float}
\usepackage{amsmath}
\usepackage{amsthm}
\usepackage{todonotes}

\usetikzlibrary{shapes.geometric}
 \newcommand{\withCol}[1]{#1} \newcommand{\noCol}[1]{}
\newcommand{\2}{\vspace{2mm}}
\newcommand{\be}{\begin{enumerate}}
\newcommand{\ee}{\end{enumerate}}
\newcommand{\bd}{\begin{description}}
\newcommand{\ed}{\end{description}}

\newcommand{\beq}{\begin{equation}}
\newcommand{\eeq}{\end{equation}}

\usetikzlibrary{shapes,snakes}

\renewenvironment{proof}[1][]{\par \noindent {\bf Proof#1}.\ }{\hfill$\Box$
\par \vspace{11pt}}

\newtheorem{theorem}{Theorem}[section]
\newtheorem{lemma}[theorem]{Lemma}
\newtheorem{proposition}[theorem]{Proposition}
\newtheorem{corollary}[theorem]{Corollary}
\newtheorem{claim}{Claim}

\theoremstyle{definition}
\newtheorem{definition}[theorem]{Definition}
\newtheorem{conjecture}[theorem]{Conjecture}
\newtheorem{question}[theorem]{Question}
\newtheorem{problem}[theorem]{Problem}

\newtheorem{alg}[theorem]{Algorithm}

\newcommand{\pf}{{\bf Proof: }}
\newcommand{\JBJ}[1]{{#1}}
\newcommand{\jbj}[1]{{#1}}
\newcommand{\AY}[1]{{#1}}
\newcommand{\FP}[1]{{#1}}

\tikzstyle{vertexX}=[circle,draw, top color=gray!10, bottom color=gray!70, minimum size=14pt, scale=0.6, inner sep=0.1pt]
\tikzstyle{vertexY}=[circle,draw, top color=green!10, bottom color=green!70, minimum size=14pt, scale=0.6, inner sep=0.1pt]
\tikzstyle{vertexZ}=[circle,draw, top color=orange!10, bottom color=orange!70, minimum size=14pt, scale=0.6, inner sep=0.1pt]
\begin{document}

\title{Vertex-partitions of 2-edge-colored graphs\thanks{Research supported by the Independent
    Research Fund Denmark under grant number DFF 7014-00037B}}
\author{J\o{}rgen Bang-Jensen\thanks{Department of Mathematics and Computer
    Science, University of Southern Denmark, Odense, Denmark (email:
    jbj@imada.sdu.dk)} \and Francois Pirot\thanks{Universit\'e Paris-Saclay, France (email francois.pirot@universite-paris-saclay.fr)}\and 
Anders Yeo\thanks{Department of
    Mathematics and Computer Science, University of Southern Denmark,
    Odense, Denmark and  Department of Mathematics, University of Johannesburg, Auckland Park, 2006 South Africa (email: yeo@imada.sdu.dk).}}

\maketitle
\begin{abstract}
  A {\bf $\mathbf{k}$-majority coloring} of a digraph $D=(V,A)$ is a coloring  of $V$ with $k$ colors so that each vertex $v\in V$ has at least as many out-neighbours of color different from its own color as it has out-neighbours with the same color as itself.
  Majority colorings have received much attention in the last years and many interesting open problems remain.
  Inspired by this and the fact that digraphs can be modelled via 2-edge-colored graphs we study several problems concerning vertex partitions of 2-edge-colored graphs. \JBJ{In particular we study vertex partitions with the property that for each $c=1,2$ every vertex $v$ has least as many edges of colour $c$ to vertices outside the set it belongs to as it has to vertices inside its own set. We call such a vertex partition with $k$ sets a {\bf $\mathbf{k}$-majority partition.} Among other things we show that every 2-edge-coloured graph has a 4-majority partition and that it is NP-complete to decide whether a 2-edge-coloured graph has a 3-majority partition. We also apply probabilistic tools to show that every  $2$-edge-colored graph $G$ of minimum color-degree $\delta$ and maximum degree $\Delta \le \frac{e^{\delta/18}}{9\delta}-2$ has a balanced majority  $3$-partition.

  }
\end{abstract}

\section{Introduction}

The following well-known fact was first observed by  Erd\H{o}s~\cite{erdosIJM3}. It is also very easy to produce such a spanning bipartite subgraph $H$ satisfying (\ref{bip2part}). 

\begin{proposition}[Erd\H{o}s~\cite{erdosIJM3}]
\label{prop:bipsp}
Every graph $G=(V,E)$ has a spanning bipartite subgraph $H$ such that
\begin{equation}
  \label{bip2part}
  d_H(v)\geq \frac{1}{2}d_G(v) \hspace{2mm}\forall v\in V
  \end{equation}
\end{proposition}

A {\bf majority \textbf{k}-coloring} of a (di)graph $G$ is a coloring $c:V(G)\rightarrow \{1,\ldots{},k\}$ of the vertices of $G$ satisfying that for every $i\in [k]$ each vertex of color $i$ has at least as many (out-)neighbours of color different from $i$ as it has (out-)neighbours of color $i$. By Proposition \ref{prop:bipsp} each undirected graph $G$ has a majority coloring with only 2 colors.

It is easy to see that not all digraphs have a majority coloring with only 2 colors. The smallest such example is the directed 3-cycle. In fact the following holds.

\begin{theorem}\cite{bangTCS719}
  It is NP-complete to decide whether a given digraph has a majority coloring with 2 colors.
\end{theorem}

It is also easy to see that every acyclic digraph $D=(V,A)$ has a majority coloring with 2 colors: Let $v_1,\ldots{},v_n$ be an acyclic ordering of $V$, that is, every arc in $A$ is of the form $v_iv_j$ where $j>i$. Color $v_n$ aribitrarily and after coloring $v_{i+1},\ldots{},v_n$ we chose the color of $v_i$ so that it has at least half of its out-neighbours of the opposite color. This and the fact that we can partition the arcs of any digraph into two sets each inducing an acyclic subdigraph easily implies the following.

\begin{theorem}\cite{kreutzerEJC24}
  \label{thm:majority4}
  Every digraph has a majority coloring with 4 colors
\end{theorem}

In  \cite{kreutzerEJC24} a number of results on majority colorings of digraphs were obtained.
  While the proof of Theorem \ref{thm:majority4} is very simple, the following conjecture seems much more difficult and is still widely open, even for well structured classes of digraphs such as tournaments and eulerian digraphs.

\begin{conjecture}\cite{kreutzerEJC24}
  \label{conj:3colsD}
  Every digraph has a majority coloring with 3 colors.
\end{conjecture}

The conjecture was verified for digraphs of high out-degree in \cite{kreutzerEJC24}
\begin{theorem}\cite{kreutzerEJC24}
  Every digraph on $n$ vertices and minimum out-degree larger than $72\log{}(3n)$ has a majority 3-coloring.
\end{theorem}

If also the maximum in-degree is bounded, then the conjecture holds for a more modest lower bound on the out-degree.

\begin{theorem}\cite{kreutzerEJC24}
  Every digraph on $n$ vertices which has minimum out-degree $\delta^+\geq 1200$ and maximum in-degree at most $\frac{e^{(\delta^+/72)}}{12\delta^+}$ has a majority 3-coloring.
\end{theorem}

A number of other results on majority colorings of digraphs and Conjecture \ref{conj:3colsD} have been obtained, see e.g. \cite{anastosEJC28, anholcerDM348, giraoCPC26,jiICMAI}.

  The focus of this paper is on vertex partitions of 2-edge-colored graphs. It is well known that many problems for  digraphs such as the hamiltonian cycle problem as well as problems concerning paths and cycles in bipartite digraphs are specical cases of problems for bipartite 2-edge-colored graphs (see e.g. \cite[Chapter 16]{bang2009}). In fact, if we replace each arc $uv$ in a digraph $D=(V,A)$ by the two new edges $ux_{uv},x_{uv}v$ such that the first edge is red and the second blue and call the resulting 2-edge-coloured graph $G$, then every directed cycle $C$ in $D$ corresponds to a cycle $C'$ in $G$ where the colours of the edges alternate between red and blue. Hence the results on majority colourings of digraphs make it interesting to consider analogs for 2-edge-colored graphs. In order to avoid confusion between the colors of the edges and a vertex coloring, we use the
  equivalent notion of vertex partitions.

Let $G=(V,E)$ be a graph and let $c:E\rightarrow \{1,2\}$ be a coloring  of the edges of $E$ by 2 colors. A $k$-partition of $V$ is a partition $V=V_1\cup\ldots\cup{}V_k$ of $V$ into $k$ disjoint sets.
Such a partition is called a {\bf majority partition} if the following holds for every $c\in [2]$ and $i\in [k]$, where $d_{c,X}(v)$ denotes the number of edges $vx$ of color $c$ between $v$ and the set $X$.
\begin{equation}
  d_{c,V_i}(v)\leq d_{c,V-V_i}(v) \hspace{2mm} \forall v\in V_i
\end{equation}

Thus in a majority partition, for each color $c\in [2]$, every vertex has at least as many edges of color $c$ to the sets not containing $v$ as it has to vertices inside the set containing it. \\

Note that, just as it is the case for undirected graphs, an edge between two vertices $x,y$ in different parts of a $k$-partition
counts as a good neighbour for both $x$ and $y$. This is not the case for digraphs where an arc from $x$ to $y$ only helps $x$. Hence in some sense it is natural to think that an analogue for Conjecture \ref{conj:3colsD} should hold for majority partitions of 2-edge-colored graphs, namely that every 2-edge-colored graph has a majority partition with 3 sets. As we will see in Section \ref{sec:no3part}, this turns out to be false in a strong sense, but the following analogue of Theorem \ref{thm:majority4}
is easy to establish.

\begin{proposition}
  Every 2-edge-colored graph $G$ has a majority partition with 4 sets.
\end{proposition}
\pf Let $G_i=(V,E_i)$ denote the spanning subgraph of $G$ induced by the edges of color $i$ for $i=1,2$.
By Proposition \ref{prop:bipsp} $G_1$ contains a spanning bipartite subgraph $H_1=(V_1,V_2,E'_1)$ such that
$d_{H_1}(v)\geq \frac{d_{G_1}(v)}{2}$ for every $v\in V$ and $G_2$ contains a spanning bipartite subgraph $H_2=(W_1,W_2,E'_2)$ such that
$d_{H_2}(v)\geq \frac{d_{G_2}(v)}{2}$ for every $v\in V$. Now the 4-partition $V_{11},V_{12},V_{21},V_{22}$, where $V_{ij}=V_i\cap W_j$ for $i,j\in [2]$ is a majority partition of $G$ wrt. the coloring $c$. \qed\\

The smallest 2-edge-colored graph with no majority 2 partition is obtained by taking a non-monochromatic 3-cycle (both colors used). Figure \ref{fig:not2} shows another example which is both complete and eulerian (the red and the blue degrees are equal at every vertex).

\begin{figure}[H]

  \begin{center}
    \tikzstyle{vertexB}=[circle,draw, minimum size=15pt, scale=0.8, inner sep=0.5pt]
\tikzstyle{vertexBsmall}=[circle,draw, minimum size=10pt,  top color=gray!50, bottom color=gray!10, scale=0.8, inner sep=0.5pt]
\begin{tikzpicture}[scale=0.6]
  \node (x1) at (3.5,1) [vertexB]{$x_1$};
  \node (x2) at (2,4) [vertexB]{$x_2$};
  \node (x3) at (5,6) [vertexB]{$x_3$};
  \node (x4) at (8,4) [vertexB]{$x_4$};
  \node (x5) at (6.5,1) [vertexB]{$x_5$};
  \draw [color=blue] (x1) to (x2);
  \draw [color=blue] (x2) to (x3);
  \draw [color=blue] (x3) to (x4);
  \draw [color=blue] (x4) to (x5);
  \draw [color=blue] (x1) to (x5);
  \draw[color=red] (x1) to (x3);
  \draw[color=red] (x5) to (x3);
  \draw[color=red] (x2) to (x4);
  \draw[color=red] (x1) to (x4);
  \draw[color=red] (x5) to (x2);

  \end{tikzpicture}
\end{center}
\caption{A 2-edge-colored complete graph which has no majority partition with two sets.}\label{fig:not2}
  \end{figure}
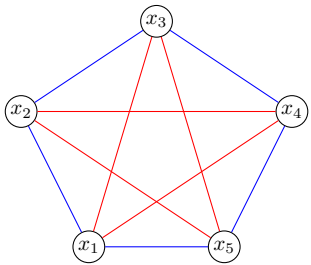

  To see that the 2-edge-coloured complete graph in Figure \ref{fig:not2} does not have a majority 2-partition, suppose that $V_1,V_2$ is a majority partition of $\{x_1,x_2,x_3,x_4,x_5\}$. We may assume that $x_1\in V_1$ and $x_2\in V_2$. If $x_3\in V_2$, then we must have $x_4\in V_1$ or $x_3$ would have no blue edge to $V_1$. This forces $x_5\in V_2$ as both of its blue neighbours are in $V_1$ but now $x_5$ has no red edge to $V_1$. Hence we conclude that $x_3\in V_1$ and then $x_4\in V_2$ as $x_1$ must have a red edge to $V_2$. This forces $x_5\in V_1$ as $x_2$ must have a red edge to $V_1$. But now $x_3$ has no red neighbour in $V_2$, contradiction.\\

  \JBJ{This paper is organized as follows.  In Section \ref{sec:2part} we prove that it is NP-complete to decide whether a 2-edge-coloured graph has a majority 2-partition. We also prove that already to decide the existence of a 2-partition where each vertex has edges of both colours to the other set is NP-complete. In Section
    \ref{sec:no3part} we give a construction of a 2-edge coloured graph that has no majority 3-partition and based on this construction we show in Section \ref{sec:NPC3part} that it is NP-complete to decide whether a 2-edge-coloured graph has a majority 3-partition. In Section \ref{sec:proba} we use probabilistic tools to prove some results on balanced (majority) partitions of 2-edge-coloured graphs. In Section \ref{sec:spanning} we first prove that no matter how high the edge-connectivity of the subgraphs induced by the two edge colors of a 2-edge-coloured graph $G$ is, there may not exist any 2-partition of $G$ which is connected in both colors. Then we prove that it is NP-complete to decide whether a given 2-edge-coloured graph with edges $E=E_1\cup E_2$, where $E_i$ are the edges  of colour $i$ has a 2-partition $V_1,V_2$ such that the bipartite spanning subgraph $B_i=(X_1,X_2,E'_i)$ is connected for $i=1,2$ where $E'_i$ consists of those  edges of $E_i$ with precisely one end in $X_1$.
    Finally in Section \ref{sec:remarks} we discuss some open problems related to majority partitions of 2-edge-colored graphs.
}

%\section{Preliminaries}

\section{(Majority) partitions with 2 sets}\label{sec:2part}

In this section we denote the two colors of a 2-edge-coloring by red and blue. The first result deals with 2-partitions so that each vertex has at least one red and one blue edge to the other part.
\begin{theorem}
  \label{thm:atleastone}
  It is NP-complete to decide whether a given 2-edge-colored graph $G$ has spanning bipartite subgraph $B$ such that
  \begin{equation}
    \label{goodbip}
    \min\{d_{1,B}(v),d_{2,B}(v)\}\geq 1 \hspace{2mm}\forall{} v\in V
    \end{equation}
  \end{theorem}
\pf We reduce from 3-SAT so let ${\cal F}=C_1\wedge{}C_2\wedge\ldots\wedge{}C_m$ be a in instance of 3-SAT over the variables $x_1,\ldots{},x_n$. From this instance we construct a graph $G=(V,E)$ and a 2-coloring $c:E\rightarrow \{red,blue\}$ as follows: The vertex set of $G$ consists of
\begin{itemize}
\item Two vertices $w_j,w'_j$ for each $j\in [m]$ (corresponding to the $j$th clause in $\cal F$).
\item Two vertices $v_i,\bar{v}_i$ for each $i\in [n]$ (corresponding to the two litterals over the variable $x_i$).
  \item Eight extra vertices $z_1,z_2,z_3,z_4,h_1,h_2,h_3,h_4$
  \end{itemize}

  The edge set of $G$ is

  \begin{itemize}
  \item red edges $\{w'_jw_j| j\in [m]\}$
  \item red edges $\{v_i\bar{v}_i|i\in [n]\}$
  \item red edges $z_2z_3,z_4z_1,h_1h_2,h_3h_4$
   \item blue edges $z_1z_2,z_3z_4,h_2h_3,h_4h_1$
   \item blue edges $\{z_1w'_j|j\in [m]\}$
     \item blue edges $\{h_1v_i,h_1\bar{v}_i|i\in [n]\}$.
  \item For each $j\in [m]$: and $i\in [n]$: if $x_i$ ($\bar{x}_i$)
    is a literal of $C_j$ then $w_jv_i$ ($w_j\bar{v}_i$) is a blue edge.
  \end{itemize}

  This completes the description of $G$ which can clearly be constructed from $\cal F$ in polynomial time. See Figure \ref{fig1} for an example.\\
  
  We claim that $G$ has a spanning bipartite subgraph satisfying (\ref{goodbip}) if and only if $\cal F$ is satisfiable.\\

  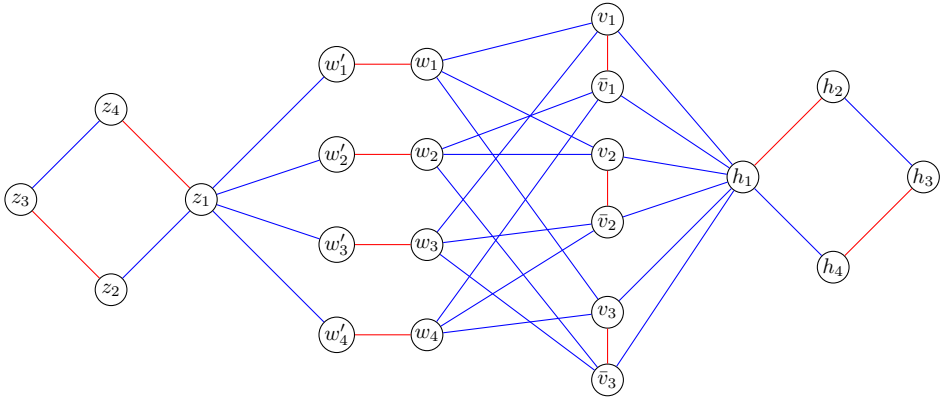
\begin{figure}[H]
    \begin{center}
      \tikzstyle{vertexB}=[circle,draw, minimum size=15pt, scale=0.8, inner sep=0.5pt]
\tikzstyle{vertexBsmall}=[circle,draw, minimum size=10pt,  top color=gray!50, bottom color=gray!10, scale=0.8, inner sep=0.5pt]
\begin{tikzpicture}[scale=0.6]
  
  \node(w'1) at (5,10)[vertexB] {$w'_1$};
  \node(w1) at (7,10)[vertexB] {$w_1$};
  \draw [color=red] (w'1) to (w1);

  \node(w'2) at (5,8)[vertexB] {$w'_2$};
  \node(w2) at (7,8)[vertexB] {$w_2$};
  \draw [color=red] (w'2) to (w2);

  \node(w'4) at (5,4)[vertexB] {$w'_4$};
  \node(w4) at (7,4)[vertexB] {$w_4$};
  \draw [color=red] (w'4) to (w4);

   \node(w'3) at (5,6)[vertexB] {$w'_3$};
  \node(w3) at (7,6)[vertexB] {$w_3$};
  \draw [color=red] (w'3) to (w3);

  \node(x1) at (11,11) [vertexB] {$v_1$};
  \node(x1b) at (11,9.5)[vertexB] {$\bar{v}_1$};
  \draw [color=red] (x1) to (x1b);

  \node(x2) at (11,8) [vertexB] {$v_2$};
  \node(x2b) at (11,6.5)[vertexB] {$\bar{v}_2$};
  \draw [color=red] (x2) to (x2b);

  \node(x3) at (11,4.5) [vertexB] {$v_3$};
  \node(x3b) at (11,3)[vertexB] {$\bar{v}_3$};
  \draw [color=red] (x3) to (x3b);

  \node (b1) at (14,7.5) [vertexB] {$h_1$};
  \node (b2) at (16,9.5) [vertexB] {$h_2$};
  \node (b3) at (18,7.5) [vertexB] {$h_3$};
  \node (b4) at (16,5.5) [vertexB]{$h_4$};
  \draw [color=red] (b1) to (b2);
  \draw [color=red] (b3) to (b4);
  \draw [color=blue] (b1) to (b4);
  \draw [color=blue] (b3) to (b2);

  \node (z3) at (-2,7) [vertexB] {$z_3$};
  \node (z4) at (0,9) [vertexB] {$z_4$};
  \node (z1) at (2,7) [vertexB] {$z_1$};
  \node (z2) at (0,5) [vertexB]{$z_2$};
  \draw [color=red] (z1) to (z4);
  \draw [color=red] (z3) to (z2);
  \draw [color=blue] (z1) to (z2);
  \draw [color=blue] (z3) to (z4);

  \draw [color=blue] (z1) to (w'1);
  \draw [color=blue] (z1) to (w'2);
  \draw [color=blue] (z1) to (w'3);
  \draw [color=blue] (z1) to (w'4);

  \draw [color=blue] (w1) to (x1);
  \draw [color=blue] (w1) to (x2);
  \draw [color=blue] (w1) to (x3);

  \draw [color=blue] (w2) to (x1b);
  \draw [color=blue] (w2) to (x2);
  \draw [color=blue] (w2) to (x3b);

  \draw [color=blue] (w3) to (x1);
  \draw [color=blue] (w3) to (x2b);
  \draw [color=blue] (w3) to (x3b);

  \draw [color=blue] (w4) to (x1b);
  \draw [color=blue] (w4) to (x2b);
  \draw [color=blue] (w4) to (x3);

  \draw [color=blue] (x1) to (b1);
  \draw [color=blue] (x2) to (b1);
  \draw [color=blue] (x3) to (b1);
  \draw [color=blue] (x1b) to (b1);
  \draw [color=blue] (x2b) to (b1);
  \draw [color=blue] (x3b) to (b1);

\end{tikzpicture} \hfill

    \end{center}
    \caption{The 2-edge-colored graph $G$ corresponding to the 3-SAT formula ${\cal F}= (x_1\lor{}x_2\lor{}x_3)\wedge{}(\bar{x}_1\lor{}x_2\lor{}\bar{x}_3)\wedge{}(x_1\lor{}\bar{x}_2\lor{}\bar{x}_3)\wedge{}(\bar{x}_1\lor\bar{x}_2\lor{}x_3)$}\label{fig1}
  \end{figure}

  For each $j\in [m]$ let $W_j$ denote the set of the 3 neighbors of $w_j$ in $\{v_1,\ldots{},v_n\}\cup\{\bar{v}_1,\ldots{},\bar{v}_n\}$. It is easy to check that the following claim is true.\\

  \begin{claim}
    \label{cl:oppositesets}
    In any spanning bipartite subgraph $B=(V_1,V_2,E')$ of $G$ which satisfies (\ref{goodbip}) the following holds:

  \begin{itemize}
    \item The vertices $w_j$ and $w'_j$ must belong to opposite sets in the bipartition of $B$ for every $j\in [m]$
  and similarly
  $v_i$ and $\bar{v}_i$ must belong to opposite sets in the bipartition of $B$ for every $i\in [n]$.
\item $z_1,z_3$ and $\{w_1,w_2,\ldots{},w_m\}$ are in the same set $V_i$
  \item $h_1$ and $h_3$ are in the same set and $h_2,h_4$ are in the other set.
  \item For each $j\in [m]$ we have $W_j\cap V_{3-i}\neq\emptyset$, where $V_i$ is the set
    containing $w_j$.\qed
  \end{itemize}
  \end{claim}

  Suppose first that $G$ has a spanning bipartite subgraph $B=(V_1,V_2,E')$ of $G$ which satisfies (\ref{goodbip}). Let $V_i$ be the set containing all $w_j$ vertices. Then we set the variable $x_i$ to true precisely if $v_i\in V_{3-i}$. By the observations  above this is a valid truth assignment and each clause is satisfied as $W_j\cap V_{3-i}\neq\emptyset$ for $j\in [m]$.\\
  Suppose now that $\phi{}$ is a satisfying truth assignment of $\cal F$. Let $V_1,V_2$ be the 2-partition defined by letting $V_2$ contain
  \begin{itemize}
  \item those vertices $v_i$ for which $\phi{}(x_i)=true$
  \item those vertices $\bar{v}_i$ which $\phi{}(x_i)=false$
  \item All vertices of $\{w'_1,w'_2,\ldots{},w'_m\}$
    \item The vertices $\{z_2,z_4,h_1,h_3\}$
    \end{itemize}
    and let $V_1=V-V_2$.

    By construction every vertex is incident to  a red edge going between $V_1$ and $V_2$ and
    it is easy to see that all vertices in $\{z_1,z_2,z_3,z_4,h_1,h_2,h_3,h_4\}\cup \{w'_1,\ldots{},w'_m\}$ and are incident to a blue edge to the opposite set. As $\phi$ makes at least one literal of $C_j$ true for each $j\in [m]$ we have $W_j\cap V_2\neq\emptyset$ for $j\in [m]$ so all vertices in $\{w_1,\ldots{},w_m\}$ have a blue edge to $V_1$. Finally every vertex in $\{v_1,\ldots{},v_n\}\cup\{\bar{v}_1,\ldots{},\bar{v}_n\}$ has a blue edge to $h_1\in V_2$, implying that the remaining vertices in $V_1$ also have a blue edge to $V_2$. As every literal of $\cal F$ appears in some clause we finally conclude that every vertex in $V_2\cap (\{v_1,\ldots{},v_n\}\cup\{\bar{v}_1,\ldots{},\bar{v}_n\})$ has a blue edge to $V_1$. This shows that the bipartite graph induced by the partition $V_1,V_2$ satisfies (\ref{goodbip}) \qed\\

    A 2-partition in which every vertex has at least one red and one blue edge to the other part is not necessarily a majority partition but the result below shows that it is also hard in general to obtain such a partition

    \begin{theorem}
      \label{thm:2partNPC}
      It is NP-complete to decide whether a given 2-edge-colored graph $G$ has a majority partition with two sets.
    \end{theorem}

    \pf Again we reduce from 3-SAT. Let ${\cal F}=C_1\wedge{}C_2\wedge\ldots\wedge{}C_m$ be a in instance of 3-SAT over the variables $x_1,\ldots{},x_n$. First note that the graph $G$ constructed in the previous proof will not allow a majority 2-partition as soon as some variable $x_i$ occurs at least twice as the literal $x_i$ and at least twice as the literal $\bar{x}_i$. This follows from the fact that all the vertices $w_1,w_2,\ldots{},w_m$ must belong to the same set of every majority 2-partition (by the same argument as we used in the proof above). In particular the graph in Figure \ref{fig1} has no majority 2-partition. We now show how to change the graph $G$ from the proof above to a new 2-edge-colored graph $G^*$ such that this has a majority partition with two sets if and only if $\cal F$ is satisfiable.
    For each $i\in [n]$ let $q_i$ be the maximum of the number of occurences of $x_i$ in clauses and the number of occurences of $\bar{x}_i$ in clauses and change $G$ as follows:

    \begin{itemize}
   
      \item Delete the blue edges between $\{v_1,\ldots{},v_n\}\cup\{\bar{v}_1,\ldots{},\bar{v}_n\}$ and $h_1$.
    \item Introduce new vertices $\{w''_j|j\in [m]\}$ and add blue edges $\{w''_jw_j|j\in [m]\}$ and red edges $\{w''_jz_1|j\in [m]\}$.
      \item For each $i\in [n]$ and $r\in [q_i]$ add a copy  $H_{i,r}$ of the subgraph of $G$ induced by the vertices $\{h_1,h_2,h_3,h_4\}$ (it is an alternating 4-cycle) and join $v_i$ by a blue edge to the vertex corresponding to $h_1$ in each of the subgraphs $H_{i,1},\ldots{},H_{i,q_i}$. Similarly, for each $i\in [n]$ and $r\in [q_i]$ add a copy  $\bar{H}_{i,r}$ of the subgraph of $G$ induced by the vertices $\{h_1,h_2,h_3,h_4\}$ and join $\bar{v}_i$ by a blue edge to the vertex corresponding to $h_1$ in each of the subgraphs $\bar{H}_{i,1},\ldots{},\bar{H}_{i,q_i}$.
      \end{itemize}

      Note that, as it was the case for $G$, the red edges form a perfect matching of $G^*$ and hence the two vertices of each red edge must belong to different sets in any majority partition with two sets.

      It is not difficult to check that Claim \ref{cl:oppositesets} also holds for $G^*$ where we let $h_1,h_2,h_3,h_4$ refer to arbitrary copies of those vertices in $G^*$. Here we used that the blue degree of each vertex $w_j$, $j\in [m]$ is 4. Thus the argument that $\cal F$ is satisfiable if $G^*$ has a majority partition with 2 sets is analogous to the first part of the proof of Theorem \ref{thm:atleastone}.

      Suppose now that $\phi$ is a satisfying truth assignment for $\cal F$. Then we define the following 2-partition $V_1,V_2$ of $G^*$.
      Start by letting $V_1=\emptyset=V_2$

      \begin{itemize}
        \item For those vertices $v_i$ for which $\phi{}(x_i)=true$: add $v_i$ to $V_2$, add all vertices corresponding to $h_1,h_3$ in $H_{i,r}$, $r\in [q_i]$ to $V_1$ and all vertices corresponding to $h_2,h_4$ in $H_{i,r}$ to $V_2$.
  \item For those vertices $\bar{v}_i$ which $\phi{}(x_i)=false$: add $\bar{v}_i$ to $V_2$, add all vertices corresponding to $h_1,h_3$ in $\bar{H}_{i,r}$, $r\in [q_i]$ to $V_1$ and all vertices corresponding to $h_2,h_4$ in $\bar{H}_{i,r}$ to $V_2$.
  \item Add all vertices of $\{w'_1,w'_2,\ldots{},w'_m\}$, $\{w''_1,\ldots{},w''_m\}$ and $\{z_2,z_4\}$ to $V_2$.
    \item add all vertices not placed yet to $V_1$.
    \end{itemize}

    Let $B^*$ be the bipartite spanning subgraph of $G^*$ induced by the partition $V_1,V_2$.
    It is easy to check that  all vertices in every copy of $H$ as well as all vertices in
    $\{z_1,z_2,z_3,z_4\}\cup\{w'_1,w'_2,\ldots{},w'_m\}\cup \{w''_1,\ldots{},w''_m\}$ have both blue and red degree in $B^*$ at least half of their red/blue degree in $G^*$. Furthermore each vertex in $w_j$, $j\in [m]$ belongs to $V_1$ and has its only red edge going to $V_2$ and at least two of its four blue edges go to a vertex in $V_2$ because $w''_j\in V_2$ and since $\phi$ is a satisfying truth assignment, at least one vertex of $W_j$ is also in $V_2$. Finally every vertex in $\{v_1,\ldots{},v_n\}\cup\{\bar{v}_1,\ldots{},\bar{v}_n\}$ has the same red degree (one) in $B^*$ as in $G^*$ and the  blue degree of $v_i$, $\bar{v}_i$ in $B^*$ is at least $q_i$ which, by the definition of the $q_i$'s, is at least half the blue degree of $v_i$, $\bar{v}_i$ in $G^*$. This shows that $V_1,V_2$ is a majority partition of $G^*$.\qed

    \section{A 2-edge-colored graph with no majority 3-partition}\label{sec:no3part}

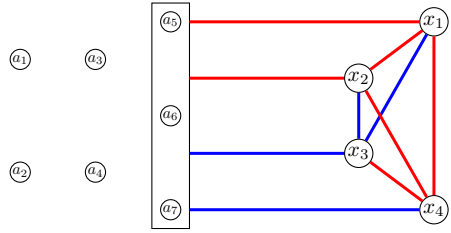
\begin{figure}[th]
  \begin{center}

\tikzstyle{vertexL}=[circle,draw, minimum size=8pt, scale=0.6, inner sep=0.5pt]
\tikzstyle{vertexB}=[circle,draw, minimum size=10pt, scale=0.8, inner sep=0.6pt]
\tikzstyle{vertexR}=[circle,draw, color=red!100, minimum size=14pt, scale=0.6, inner sep=0.5pt]
\begin{tikzpicture}[scale=0.5]
% \node at (8,2.7) {$2$};

  \node (a1) at (1,5) [vertexL]{$a_1$};
  \node (a2) at (1,2) [vertexL]{$a_2$};
  \node (a3) at (3,5) [vertexL]{$a_3$};
  \node (a4) at (3,2) [vertexL]{$a_4$};
  \node (a5) at (5,6) [vertexL]{$a_5$};
  \node (a6) at (5,3.5) [vertexL]{$a_6$};
  \node (a7) at (5,1) [vertexL]{$a_7$};
  \draw (4.5,0.5) rectangle (5.5,6.5);

  \node (x1) at (12,6) [vertexB]{$x_1$};
  \node (x2) at (10,4.5) [vertexB]{$x_2$};
  \node (x3) at (10,2.5) [vertexB]{$x_3$};
  \node (x4) at (12,1) [vertexB]{$x_4$};

\withCol{{\color{blue}
  \draw[line width=0.04cm] (5.5,2.5) to (x3);
  \draw[line width=0.04cm] (5.5,1) to (x4);
  \draw[line width=0.04cm] (x1) to (x3);
  \draw[line width=0.04cm] (x2) to (x3);
}}

\noCol{
  \draw[line width=0.07cm] (5.5,2.5) to (x3);
  \draw[line width=0.07cm] (5.5,1) to (x4);
  \draw[line width=0.07cm] (x1) to (x3);
  \draw[line width=0.07cm] (x2) to (x3);
}

\withCol{{\color{red}
  \draw[line width=0.04cm] (5.5,6) to (x1);
  \draw[line width=0.04cm] (5.5,4.5) to (x2);
  \draw[line width=0.04cm] (x1) to (x2);
  \draw[line width=0.04cm] (x1) to (x4);
  \draw[line width=0.04cm] (x2) to (x4);
  \draw[line width=0.04cm] (x3) to (x4);
}}

\noCol{
  \draw[line width=0.02cm] (5.5,6) to (x1);
  \draw[line width=0.02cm] (5.5,4.5) to (x2);
  \draw[line width=0.02cm] (x1) to (x2);
  \draw[line width=0.02cm] (x1) to (x4);
  \draw[line width=0.02cm] (x2) to (x4);
  \draw[line width=0.02cm] (x3) to (x4);
}

% \node at (1,5) {$T_1$};
   \end{tikzpicture} 
\caption{The gadget $H_{5,6,7}$ and its connection to $\{a_5,a_6,a_7\}$.} \label{fig:gadgetH}
\end{center} \end{figure}

We will now construct a $2$-edge-colored graph, $G$, which does not have a majority partition with 3 sets.
We start by letting $V(G)=\{a_1,a_2,\ldots,a_7\}$.  For all integers $i$, $j$ and $k$, which satisfies $1 \leq i < j < k \leq 7$, 
we will add a gadget $H_{i,j,k}$ defined as follows. 

Let $V(H_{i,j,k})=\{x_1,x_2,x_3,x_4\}$ and let $E(H_{i,j,k})=E_1 \cup E_2$, 
where $E_1=\{x_1x_2,x_4x_1,x_4x_2,x_4x_3\}$ and   $E_2=\{x_3x_1,x_3x_2\}$.
  Let all edges in $E_1$ have color 1 (illustrated by \noCol{thin}\withCol{red} edges in Figure~\ref{fig:gadgetH}) and
let edges in $E_2$ have color 2 (illustrated by \noCol{thick}\withCol{blue} edges in Figure~\ref{fig:gadgetH}).
Finally, add all edges from $\{x_1,x_2\}$ to $\{a_i,a_j,a_k\}$ of color 1 (\noCol{thin}\withCol{red}) and
all edges from $\{x_3,x_4\}$ to $\{a_i,a_j,a_k\}$ of color 2 (\noCol{thick}\withCol{blue}).
Once we have added $H_{i,j,k}$ in the above manner for all $i$, $j$ and $k$ ($1 \leq i < j < k \leq 7$) we have our $G$,
which we shall show has no majority partition with 3 sets.

Assume for the sake of contradiction that there is a partition $(X_1,X_2,X_3)$ of $V(G)$ such that every vertex in $X_i$ has at least as many edges
to $V(G)\setminus X_i$ as it does to $X_i$ for each color and for all $i=1,2,3$. 
If $x \in X_i$ we say that $x$ gets label $i$ and indicate this by $l(x)=i$.
We will now prove the following claims, where we call edges of color 1 for \noCol{thin}\withCol{red} edges and edges of color
2 for \noCol{thick}\withCol{blue} edges (so it is easier to refer to Figure~\ref{fig:gadgetH}). Note that Claim~C implies 
a contradiction to $G$ having a majority partition with 3 sets, which completes our proof.

\2

{\bf Claim A:} We may without loss of generality assume that $l(a_5)=l(a_6)=l(a_7)$.

{\bf Proof of Claim A:} As there are seven vertices in $\{a_1,a_2,\ldots,a_7\}$ three of the vertices must have the same label. 
By renaming the vertices in this set we may without loss of generality assume that $l(a_5)=l(a_6)=l(a_7)$.\\

By Claim~A, we may without loss of generality assume that $l(a_5)=l(a_6)=l(a_7)=3$ (by possibly renaming the labels). Below we consider the subgraph $H_{5,6,7}$ of $G$ and consider the vertices $x_1,x_2,x_3,x_4$ of that subgraph.

\2

{\bf Claim B:} We may without loss of generality assume that $l(x_1),l(x_2),l(x_3),l(x_4) \in \{1,2\}$.

{\bf Proof of Claim B:}
 Due to the \noCol{thin}\withCol{red} edges incident with $x_1$, we note that $l(x_1) \not= 3$.
 Due to the \noCol{thin}\withCol{red} edges incident with $x_2$, we note that $l(x_2) \not= 3$.
 Due to the \noCol{thick}\withCol{blue} edges incident with $x_3$, we note that $l(x_3) \not= 3$.
 Due to the \noCol{thick}\withCol{blue} edges incident with $x_4$, we note that $l(x_4) \not= 3$.
 This completes the proof of Claim~B.

\2

{\bf Claim C:} We obtain a contradiction.

{\bf Proof of Claim C:} By Claim~B we note that $l(x_3)\in \{1,2\}$. Without loss of generality assume that $l(x_3)=1$ 

Due to Claim~C and the \noCol{thick}\withCol{blue} edges incident with $x_1$, we note that $l(x_1) = 2$.
Due to Claim~C and the \noCol{thick}\withCol{blue} edges incident with $x_2$, we note that $l(x_2) = 2$.
Due to Claim~C and the \noCol{thin}\withCol{red} edges incident with $x_3$, we note that $l(x_4) = 2$.
So $l(x_1) = l(x_2) = l(x_4) = 2$. However now $x_4$ has two \noCol{thin}\withCol{red} edges into $X_2$ but only 
one to $V(G) \setminus X_2$, a contradiction.\\

\section{NP-hardness of majority 3-partition for 2-edge-colored graphs}\label{sec:NPC3part}

\begin{theorem}
    It is NP-complete to decide whether a given 2-edge-colored graph $G$ has a majority 3-partition.
  \end{theorem}

  \begin{proof}
We will reduce from the problem of deciding if a $3$-uniform hypergraph is $3$-colorable (ie can we $3$-color the vertex set of a $3$-uniform hypergraph 
such that no edge becomes monochromatic). This problem is known to be NP-hard \cite{khotFOCS43}.   So let $F$ be a $3$-uniform hypergraph with 
$V(F)=\{v_1,v_2,\ldots,v_n\}$. We will now constuct a $2$-edge-colored graph $G$, such that $G$ has a majority partition with three sets if and
only if $F$ is $3$-colorable.

Initially let $V(G)=\{a_1,a_2,\ldots,a_n\}$.  For every edge $\{v_i,v_j,v_k\}$ ($1 \leq i < j < k \leq n$) in $F$ 
add the gadget $H_{i,j,k}$ which is defined above.

If $F$ is not $3$-colorable then any $3$-coloring of $V(F)$ will result in a monochromatic edge $\{v_i,v_j,v_k\}$. 
This is equivalent to saying that any assignment of labels from $\{1,2,3\}$ to $V(G)$ will result is $l(a_i)=l(a_j)=l(a_k)$ 
for some $i$, $j$ and $k$, where
we added the gadget $H_{i,j,k}$ to $G$. By the proof in the previous section this labeling of $a_i$, $a_j$ and $a_k$ cannot be extended
to $V(H_{i,j,k})$ without obtaining a vertex in $H_{i,j,k}$ which contradicts the majority-condition.
So if $F$ is not $3$-colorable then $G$ has no majority partition with three sets.

Conversely assume that $F$ is $3$-colorable and let $c$ be a proper coloring of $F$. 
Let $l(a_i)=c(v_i)$ for all $i=1,2,\ldots, n$. We will now show how to extend the labels to the vertices of
each $H_{i,j,k}$ that has been added. Recall that, by construction, these graphs can only share vertices from $\{a_1,a_2,\ldots{},a_n\}$.

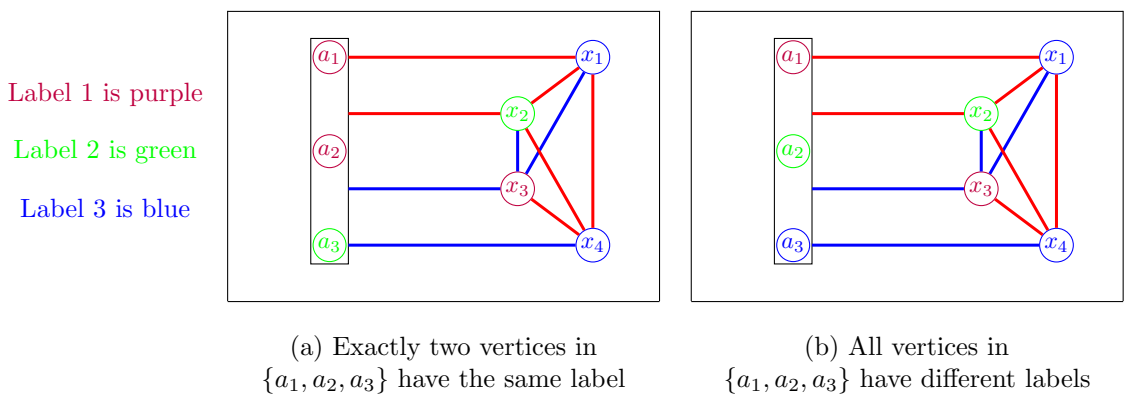
\begin{figure}[H]
  \begin{center}
\begin{tabular}{cccc} \cline{2-2} \cline{4-4} 
\begin{tikzpicture}[scale=0.5]
\node at (1,0) {\mbox{ }};
\node at (1,7) {\mbox{ }};
{\color{purple}  \node at (1,5) {Label 1 is purple};}
{\color{green}  \node at (1,3.5) {Label 2 is green};}
{\color{blue}  \node at (1,2) {Label 3 is blue};}
   \end{tikzpicture} & \multicolumn{1}{|c|}{
\tikzstyle{vertexBi}=[color=purple,circle,draw, minimum size=10pt, scale=0.9, inner sep=0.9pt]
\tikzstyle{vertexBii}=[color=green,circle,draw, minimum size=10pt, scale=0.9, inner sep=0.9pt]
\tikzstyle{vertexBiii}=[color=blue,circle,draw, minimum size=10pt, scale=0.9, inner sep=0.9pt]
\tikzstyle{vertexB}=[circle,draw, minimum size=10pt, scale=0.8, inner sep=0.6pt]
\begin{tikzpicture}[scale=0.5]
\node at (4,0) {\mbox{ }};
\node at (13,7) {\mbox{ }};
% \node at (8,2.7) {$2$};
  \node (a1) at (5,6) [vertexBi]{$a_1$};
  \node (a2) at (5,3.5) [vertexBi]{$a_2$};
  \node (a3) at (5,1) [vertexBii]{$a_3$};
  \draw (4.5,0.5) rectangle (5.5,6.5);

  \node (x1) at (12,6) [vertexBiii]{$x_1$};
  \node (x2) at (10,4.5) [vertexBii]{$x_2$};
  \node (x3) at (10,2.5) [vertexBi]{$x_3$};
  \node (x4) at (12,1) [vertexBiii]{$x_4$};

\withCol{{\color{blue}
  \draw[line width=0.04cm] (5.5,2.5) to (x3);
  \draw[line width=0.04cm] (5.5,1) to (x4);
  \draw[line width=0.04cm] (x1) to (x3);
  \draw[line width=0.04cm] (x2) to (x3);
}}

\noCol{
  \draw[line width=0.07cm] (5.5,2.5) to (x3);
  \draw[line width=0.07cm] (5.5,1) to (x4);
  \draw[line width=0.07cm] (x1) to (x3);
  \draw[line width=0.07cm] (x2) to (x3);
}

\withCol{{\color{red}
  \draw[line width=0.04cm] (5.5,6) to (x1);
  \draw[line width=0.04cm] (5.5,4.5) to (x2);
  \draw[line width=0.04cm] (x1) to (x2);
  \draw[line width=0.04cm] (x1) to (x4);
  \draw[line width=0.04cm] (x2) to (x4);
  \draw[line width=0.04cm] (x3) to (x4);
}}

\noCol{
  \draw[line width=0.02cm] (5.5,6) to (x1);
  \draw[line width=0.02cm] (5.5,4.5) to (x2);
  \draw[line width=0.02cm] (x1) to (x2);
  \draw[line width=0.02cm] (x1) to (x4);
  \draw[line width=0.02cm] (x2) to (x4);
  \draw[line width=0.02cm] (x3) to (x4);
}
   \end{tikzpicture} }
& & \multicolumn{1}{|c|}{
\tikzstyle{vertexBi}=[color=purple,circle,draw, minimum size=10pt, scale=0.9, inner sep=0.9pt]
\tikzstyle{vertexBii}=[color=green,circle,draw, minimum size=10pt, scale=0.9, inner sep=0.9pt]
\tikzstyle{vertexBiii}=[color=blue,circle,draw, minimum size=10pt, scale=0.9, inner sep=0.9pt]
\tikzstyle{vertexB}=[circle,draw, minimum size=10pt, scale=0.8, inner sep=0.6pt]
\begin{tikzpicture}[scale=0.5]
\node at (4,0) {\mbox{ }};
\node at (13,7) {\mbox{ }};
% \node at (8,2.7) {$2$};
  \node (a1) at (5,6) [vertexBi]{$a_1$};
  \node (a2) at (5,3.5) [vertexBii]{$a_2$};
  \node (a3) at (5,1) [vertexBiii]{$a_3$};
  \draw (4.5,0.5) rectangle (5.5,6.5);

  \node (x1) at (12,6) [vertexBiii]{$x_1$};
  \node (x2) at (10,4.5) [vertexBii]{$x_2$};
  \node (x3) at (10,2.5) [vertexBi]{$x_3$};
  \node (x4) at (12,1) [vertexBiii]{$x_4$};

\withCol{{\color{blue}
  \draw[line width=0.04cm] (5.5,2.5) to (x3);
  \draw[line width=0.04cm] (5.5,1) to (x4);
  \draw[line width=0.04cm] (x1) to (x3);
  \draw[line width=0.04cm] (x2) to (x3);
}}

\noCol{
  \draw[line width=0.07cm] (5.5,2.5) to (x3);
  \draw[line width=0.07cm] (5.5,1) to (x4);
  \draw[line width=0.07cm] (x1) to (x3);
  \draw[line width=0.07cm] (x2) to (x3);
}

\withCol{{\color{red}
  \draw[line width=0.04cm] (5.5,6) to (x1);
  \draw[line width=0.04cm] (5.5,4.5) to (x2);
  \draw[line width=0.04cm] (x1) to (x2);
  \draw[line width=0.04cm] (x1) to (x4);
  \draw[line width=0.04cm] (x2) to (x4);
  \draw[line width=0.04cm] (x3) to (x4);
}}

\noCol{
  \draw[line width=0.02cm] (5.5,6) to (x1);
  \draw[line width=0.02cm] (5.5,4.5) to (x2);
  \draw[line width=0.02cm] (x1) to (x2);
  \draw[line width=0.02cm] (x1) to (x4);
  \draw[line width=0.02cm] (x2) to (x4);
  \draw[line width=0.02cm] (x3) to (x4);
}
   \end{tikzpicture} } \\ \cline{2-2} \cline{4-4}
 & & & \\
 &  (a) Exactly two vertices in &  & (b) All vertices in \\
 & $\{a_1,a_2,a_3\}$ have the same label & & $\{a_1,a_2,a_3\}$ have different labels \\
 & & & \\
\end{tabular}
\caption{The gadget $H_{1,2,3}$ and proper $3$-partitions when $\{a_1,a_2,a_3\}$ are not all of the same color.} \label{fig:colorings}
\end{center} \end{figure}

Without loss of generality assume that $H_{1,2,3}$ has been added to $G$, which implies that $\{v_1,v_2,v_3\}$ was an edge in $F$
and therefore we do not have $l(a_1)=l(a_2)=l(a_3)$. We consider the cases when $\{l(a_1),l(a_2),l(a_3)\}=\{1,2,3\}$ and when this is not
the case seperately.

Assume that $\{l(a_1),l(a_2),l(a_3)\}=\{1,2,3\}$. In this case we may, without loss of generality, assume that $l(a_1)=1$, $l(a_2)=2$ and
$l(a_3)=3$. Now let $l(x_1)=3$, $l(x_2)=2$, $l(x_3)=1$ and $l(x_4)=3$. See Figure~\ref{fig:colorings}(b).

Assume that $\{l(a_1),l(a_2),l(a_3)\} \not= \{1,2,3\}$. In this case we may, without loss of generality, assume that $l(a_1)=1$, $l(a_2)=1$ and
$l(a_3)=2$. Again let $l(x_1)=3$, $l(x_2)=2$, $l(x_3)=1$ and $l(x_4)=3$. See Figure~\ref{fig:colorings}(a).

If we extend the labelings as indicated above for every $H_{i,j,k}$, then we obtain a majority partition of $V(G)$ with three sets, as desired.
So if $F$ is $3$-colorable then $G$ has a majority partition with three sets, which completes the proof.

\end{proof}

\section{Probabilistic results}
\label{sec:proba}

\subsection{Probability tools}

We begin by stating the symmetric version of the Lov\' asz Local Lemma, on which we rely for our results \JBJ{in this section}.%in Section~\ref{sec:proba}.
\begin{lemma}[Symmetric Lov\' asz Local Lemma]
\label{lem:LLL}
Consider a set $\cE=\{E_1,\dots,E_n\}$ of random (bad) events such that each $E_i$ is mutually independent of $\cE-(\cD_i\cup \{E_i\})$, for some $\cD_i \subseteq \cE$.
If there exists a real value $p\in [0,1]$ and some $D \ge 0$ such that $\pr{E_i}\le p$ and $|\cD_i|\le D$ for every $i\in [n]$, and if 
\[
epD \le 1,
\]
then the probability that none of the events in $\cE$ occur is positive.
\end{lemma}

We will also use the general form of Lov\' asz Local Lemma to obtain more precise bounds.

\begin{lemma}[General Local Lemma]
\label{lem:GLL}
Consider a set $\cE=\{E_1,\dots,E_n\}$ of random (bad) events such that each $E_i$ is mutually independent of $\cE-(\cD_i\cup \{E_i\})$, for some $\cD_i \subseteq \cE$. If there exist reals $0 < x_1, \ldots, x_n < 1$ such that for each $i$,
\begin{align*}
 \pr{E_i} \le x_i \prod_{j \in \cD_i} (1-x_j),
 \end{align*}
then the probability that none of the events in
$\cE$ occur is positive.
\end{lemma}

%
%\begin{lemma}
%\label{lem:binomial}
%Given $p\in (0,1)$, let $\Lambda_p(x) \coloneqq x \ln \frac{x}{p} + (1-x) \ln \frac{1-x}{1-p}$, for every $x\in (0,1)$.
%Let $X$ follow the binomial random distribution $\cB(n,p)$, for some integer $n$. Then
%\begin{align*}
%\pr{X \ge t} &\le e^{-n\, \Lambda_p(t/n)} \quad \et \\
%\pr{X \le t} &\le e^{-n\, \Lambda_p(t/n)},
%\end{align*}
%%
%for every integer $t\le n$.
%\end{lemma}

Finally, we need the following large concentration inequality on the sum of independent random variables given by the Hoeffding's inequalities.

\begin{lemma}[Hoeffding's inequality]
\label{lem:chernoff}
Let $X$ be \FP{the sum of $n$ independent random $0,1$-valued variables}, for some integer $n$. Then, for every $\sigma \ge 0$, 
%\begin{align*}
%\pr{X \le np-\sigma} &\le e^{-\frac{\sigma^2}{2np}},
%\end{align*}
%or
\begin{align*}
\pr{X \le \esp{X}-\sigma} &\le e^{-\frac{2\sigma^2}{n}}.
\end{align*}
\end{lemma}

\subsection{Balanced 2-partitions}

\begin{theorem}
Let $k$ be a fixed integer.
Let $G$ be a graph of maximum degree $\Delta$, and let $c$ be a  (possibly improper) $k$-edge-coloring of $G$. For every color $i\in [k]$, we assume that every vertex is incident to at least $d_i$ edges colored $i$, for some integer $d_i$.
Then there exists a balanced bipartite subgraph $H$ of $G$ such that every vertex is incident to at least $d_i/2 - O(\sqrt{d_i \ln \Delta})$ edges of $H$ colored $i$ \JBJ{for each $i\in [k]$}.
\end{theorem}

\begin{proof}
%Up to taking two disjoint copies of $G$, we may assume that $|V(G)|$ is even.
Let us first assume that $|V(G)|$ is even.
Let $M$ be a perfect matching on $V(G)$, which is not necessarily included in $E(G)$.
Let $\bphi$ be a uniformly random proper $2$-coloring of $M$, and let $\bH$ be the balanced bipartite subgraph of $G$ induced by the two color classes of $\bphi$.

Let $v\in V(G)$ be a fixed vertex. For every $i\in [k]$, we denote $N_i(v)$ the set of neighbours of $v$ in $G$ only using edges of color $i$ (that is, 
$N_i(v) = \{ u \; | \; uv \in E(G) \mbox{ and } c(uv)=i \}$).  We moreover let $\bD_i$ be the random variable that counts the number of vertices $u\in N_i(v)$ such that $\bphi(u) \neq \bphi(v)$.
For every $u\in N_i(v)$, if $uv\in M$, then by construction $u$ has a deterministic contribution of $+1$ to $\bD_i$.
We may therefore assume that we are in the worst case, i.e. the vertex matched with $v$ in $M$ is not one of its neighbours in $G$. 
Let $M_i[v]$ be the set of edges of $M$ with both endpoints in $N_i(v)$, of size $m_i(v) \coloneqq |M_i[v]|$. Let $\wN_i(v) \coloneqq N_i(v) \setminus V(M_i[v])$, and let $\wD_i$ count the number of vertices $u\in \wN_i(v)$ such that $\bphi(u) \neq \bphi(v)$. Then, since each edge $e\in M_i[v]$ covers both color classes of $\bphi$, we have
\[ \bD_i = m_i(v) + \wD_i.\] 

Let $\wn_i \coloneqq |\wN_i(v)|$ and note that by assumption we have $\wn_i \ge d_i - 2m_i(v)$. The variables $\bphi(u)$ over $u\in \wN_i(v)$ are independent, so $\wD_i$ follows the binomial distribution $\cB(\wn_i,1/2)$. So by Lemma~\ref{lem:chernoff}, for every  \AY{
$0 \le \sigma_i \le d_i/2$, the following holds (as $\wn_i/2 \ge d_i/2 - m_i(v)$),

\begin{align*}
\pr{\bD_i \le \frac{d_i}{2} - \sigma_i} & \leq   \pr{\wD_i \le \frac{d_i}{2} - \sigma_i - m_i(v)} \\ 
& \leq  \pr{\wD_i \le \frac{\wn_i}{2} - \left(\frac{\wn_i}{2} - \frac{d_i}{2} + \sigma_i + m_i(v) \right) } \\
& \leq   e^{-\frac{2(\wn_i/2 - d_i/2 + \sigma_i + m_i(v))^2}{\wn_i}}.
\end{align*}

We will now show that $e^{-\frac{2(\wn_i/2 - d_i/2 + \sigma_i + m_i(v))^2}{\wn_i}} \leq e^{-2\sigma_i^2 / d_i}$, which is equivalent to $f(\sigma_i) \geq 0$, where 
$f(\sigma_i) = \frac{(\wn_i/2 - d_i/2 + \sigma_i + m_i(v))^2}{\wn_i} - \frac{\sigma_i^2}{d_i}$. As $\wn_i/2 \ge d_i/2 - m_i(v)$ we note that this trivially holds
when $\wn_i \leq d_i$, so assume that $\wn_i > d_i$. Now $f(0) \geq 0$ and the following holds (as $\wn_i > d_i$),

\begin{align*}
f'(\sigma_i) &= \frac{2\left(\wn_i/2 - d_i/2 + \sigma_i + m_i(v) \right)}{\wn_i} - \frac{2 \sigma_i}{d_i} 
= \frac{2m_i(v)}{\wn_i} + \frac{2\sigma_i - d_i}{\wn_i} + 1 - \frac{2\sigma_i}{d_i} \\
& = \frac{2m_i(v)}{\wn_i} + (d_i - 2\sigma_i) \left(\frac{1}{d_i} - \frac{1}{\wn_i}\right)  > 0
\end{align*}

As 
$f(0) \geq 0$ and $f'(\sigma_i) \geq 0$ for all $0 \leq \sigma_i \leq d_i/2$, we infer that $f(\sigma_i) \geq 0$, and so we conclude that

\[ 
\pr{\bD_i \le \frac{d_i}{2} - \sigma_i}  \leq e^{-\frac{2(\wn_i/2 - d_i/2 + \sigma_i + m_i(v))^2}{\wn_i}}  \leq e^{- \frac{2\sigma_i^2}{d_i}}.
\]

}

Let us fix $\sigma_i = \sqrt{d_i(\AY{\ln (2ek)/2} + \ln \Delta)}$ for every $i\in [k]$.
We denote \JBJ{by} $E_i(v)$ the random (bad) even that $\bD_i \le d_i/2 - \sigma_i$. This bad event depends only on $\bphi(u)$ \AY{for $u$ where either $u \in N(v)$
or $u$ is matched to a vertex in $N(v)$ by $M$. So} in particular it is independent from any bad event $E_j(w)$ if $w$ \AY{has no path of length 1 or 2 to $v$ and also no path of length 3 where the middle edge is in $M$.}  So we may apply Lemma~\ref{lem:LLL} with 

\[D = \AY{2k\Delta^2} \quad \et \quad p = \min_{i\in [k]} e^{-2\sigma_i^2/d_i} = e^{-\AY{(\ln(2ek)} + 2\ln \Delta)} = \frac{1}{\AY{2}ek\Delta^2}.\]

We conclude that with positive probability, no bad event $E_i(v)$ occurs. So there is a realisation $H$ of $\bH$ such that every vertex is incident to at least $(d_i/2 - \sqrt{d_i(\AY{\ln (2ek)/2 + \ln \Delta)}} = d_i/2 - O(\sqrt{d_i\ln \Delta})$ edges of $H$ colored $i$.

We may now assume that $V(G)$ is odd. Let $G'$ be a graph obtained from $G$ by adding a copy $v'$ of an arbitrary vertex $v\in V(G)$. Then $G'$ has an even number of vertices and satisfies the required hypotheses; we consider a perfect matching $M$ on $V(G')$ such that $vv'\in M$, and we repeat the construction of the previous section. 
We end up with a balanced bipartite subgraph $H'$ of $G'$ such that every vertex is incident to at least $d_i/2 - \AY{\sqrt{d_i(\ln (2ek)/2 + \ln \Delta)}}$ edges of $H'$ colored $i$. Then $H \coloneqq H' \setminus \{v,v'\}$ is a balanced bipartite subgraph of $G$ such that every vertex is incident to 
at least $d_i/2 - 1 - \AY{\sqrt{d_i(\ln (2ek)/2 + \ln \Delta)}} = d_i/2 - O(\sqrt{d_i\ln \Delta})$ edges of $H$ colored $i$. This ends the proof.
\end{proof}

\subsection{Applications to balanced majority partitions}

\jbj{A $k$-partition $X_1,\ldots{},X_k$ of a set $X$ is {\bf balanced} if $||X_i|-|X_j||\leq 1$ for $i,j\in [k]$}

\begin{theorem}
\label{thm:balanced-3-partition}
Let $n$ be a multiple of $3$, and let $G$ be a 
$2$-edge-colored graph $G$ of minimum color-degree $\delta$ and maximum degree $\Delta \le \frac{e^{\delta/18}}{9\delta}$. Then $G$ has a \JBJ{balanced majority  $3$-partition}.
\end{theorem}

%qqqqqqq

\begin{proof}
First observe that we have $\delta \le \Delta \le \frac{e^{\delta/18}}{9\delta}$, which is possible only if $\delta \ge 237$; we assume this holds in what follows.

%Let us first assume that $|V(G)| = 0 \bmod 3$, and 
Let $c$ be the edge-coloring of $G$. Let $M$ be a perfect $3$-uniform matching on $V(G)$ (so a partition of $V(G)$ into triplets), and let $\bphi$ be a uniformly random rainbow $3$-coloring of $M$ (so each triplet of $M$ receives a permutation of the $3$ colors). We will show that with non-zero probability, the color classes of $\bphi$ form a \JBJ{balanced majority $3$-partition} of $G$.

Let $v\in V(G)$ be a fixed vertex. For every $i\in [2]$, we denote $N_i(v)$ the set of neighbours, \AY{$u$,} of $v$ in $G$ such that $c(uv)=i$. We moreover let $\bD_i$ be the random variable that counts the number of vertices $u\in N_i(v)$ such that $\bphi(u) \neq \bphi(v)$.
If there is a vertex matched to $v$ in $M$ that belongs to $N_i(v)$, then its contribution deterministically increases $\bD_i$ by $1$. We may therefore assume that we are in the worst case, and so that no vertex matched to $v$ in $M$ belongs to $N_i(v)$.

For \JBJ{$q\in [3]$}, we let \JBJ{$X_q$} be the set of \FP{hyperedges of $M$} with that intersect $N_i(v)$ in exactly \JBJ{$q$} vertices, of size \JBJ{$x_q \coloneqq |X_q|$}. \FP{By definition, we have $d_i(v) = 3x_3 + 2x_2 + x_1$.}
Let \JBJ{$N_i(v,q) \coloneqq N_i(v) \cap V(X_q)$}, 
and let \JBJ{$\bD_i(q)$} count the number of vertices \JBJ{$u\in N_i(v,q)$} such that $\bphi(u) \neq \bphi(v)$. 
Observe that \JBJ{$(N_i(v,q))_{q\in [3]}$} is a partition of $N_i(v)$, hence 
\begin{equation}
\label{eq:partition}
\bD_i = \bD_i(1) + \bD_i(2) + \bD_i(3).
\end{equation}
We deterministically have that \FP{$\bD_i(3) = 2x_3$}, and \FP{$\bD_i(2) = x_2 + Y$} where $Y$ is a random variable that follows the binomial distribution $\mathcal{B}(x_2,1/3)$. Finally, $\bD_i(1)$ follows the binomial distribution $\mathcal{B}(x_1,2/3)$. Using \eqref{eq:partition}, we infer that 
\begin{equation}
\label{eq:Di}
\FP{\bD_i = 2x_3 + x_2 + \wD_i,}
\end{equation}
where $\wD_i$ is \FP{the sum of $x_2$ variables following the Bernoulli distribution of parameter $1/3$ and $x_1$ variables following the Bernoulli distribution of parameter $2/3$. So we have $\esp{\wD_i}=x_2/3 + 2x_1/3$.}
We apply Lemma~\ref{lem:chernoff}, in order to obtain that for $t \le d_i(v)/2$,
\[ \pr{\bD_i \le t} = \pr{\wD_i \le t - 2x_3 - x_2} \le e^{-\frac{2\big(x_2/3+2x_1/3 - t + 2x_3 + x_2\big)^2}{x_1+x_2}} \le e^{-\frac{2(2d_i(v)/3-t)^2}{d_i(v)}} \le e^{-d_i(v)/18}.\]
We let $E_i(v)$ be the bad event that $\bD_i \le d_i/2$. 
We may now apply Lemma~\ref{lem:GLL} with the set of bad events $\{E_i(v) : i\in [3], v\in V(G)\}$, by associating to each bad event $E_i(v)$ the weight 
\[x_i(v) \coloneqq e^{\pth{\frac{1}{\delta}- \frac{1}{18}}d_i(v)}.\]
Using the well known inequality $\ln (1-z) \ge \frac{z}{1-z}$ for every $z<1$, we observe that, as soon as $\delta \ge 61$ (which is implied by our hypothesis, we have 
\begin{equation}
1 - x_i(v) \ge 1 - e^{1 - \frac{\delta}{18}} \ge \exp\pth{-{\frac{e}{1-e^{1 - \frac{\delta}{18}}}\, e^{- \frac{\delta}{18}}}}  \ge e^{-3 e^{-\frac{\delta}{18}}}.
\label{eq:lower-bound-proba}
\end{equation}
 
If we denote $M_{i,v} \subseteq M$ the set of edges that intersect $N_i(v)$, then each bad event $E_i(v)$ depends only on $\bphi(u)$ for $u\in V(M_{i,v})$. So it is independent from any bad event $E_j(u)$ if $V(M_{i,v})$ and $V(M_{j,u})$ do not intersect, i.e. if $u \notin N_j(V(M_{i,v}))$. We may therefore define $\cD_i(v)\coloneqq \{E_j(u) : j\in [3], u\in N_j(V(M_{i,v}))\} \setminus \{E_i(v)\}$, of size at most $3d_i(v)\Delta$. We have
\begin{align*}
\prod_{E_j(u) \in \cD_i(v)} \big(1-x_j(u)\big) &\ge 
e^{-3 e^{-\frac{\delta}{18}} \big|\cD_i(v)\big|} & \mbox{by \eqref{eq:lower-bound-proba}} \\
& \ge e^{-9 e^{-\frac{\delta}{18}} \, d_i(v)\Delta} \\
&\ge e^{\frac{d_i(v)}{\delta}} \ge \frac{\pr{E_i(v)}}{x_i(v)},
\end{align*}
therefore Lemma~\ref{lem:GLL} implies that there is a realisation of $\bphi$ that avoids every bad event $E_i(v)$. The conclusion follows.
\end{proof}

\begin{corollary}
let $G$ be a $2$-edge-colored graph of minimum color-degree $\delta$ and maximum degree $\Delta \le \frac{e^{\delta/18}}{9\delta}-2$. Then $G$ has a \JBJ{balanced majority  $3$-partition}.
\end{corollary}

\begin{proof}
If $|V(G)|$ is a multiple of $3$, the result is given directly by Theorem~\ref{thm:balanced-3-partition}. 

If $|V(G)|$ is not a multiple of $3$, let $G'$ be obtained from $G$ by adding $1$ or $2$ copies of an arbitrary vertex $v\in V(G)$ so that $V(G') \bmod 3=0$. Let $\Upsilon$ be the set containing $v$ and its copies in $G'$.
The minimum color-degree in $G'$ is still at least $\delta$ and $\Delta(G') \le \Delta(G)+2 \le e^{\delta/18}/(9\delta)$, so $G'$ satisfies the hypotheses of \ref{thm:balanced-3-partition}. Therefore, there is a balanced tripartite subgraph $H'$ of $G'$ such that every vertex $v\in V(G')$ is incident to more than $d'_i(v)/2$ edges of $H'$ colored $i$, for each $i\in [3]$. 
We now claim that the $3$ parts of $H \coloneqq H' \setminus \Upsilon$ form a \JBJ{balanced majority  $3$-partition} of $G$. 
Indeed, for every $u\in V(G)$, if $N(u)$ does not intersect $\Upsilon$ then $d_i(u)=d'_i(u)$ and $N_H(u)=N_{H'}(u)$, so $u$ is incident to more than $d_i(u)/2$ edges colored $i$ in $H$. If $N(u)$ intersects $\Upsilon$, then $d'_i(u) \ge d_i(u)+1$ and $\deg_H(u) \ge \deg_{H'}(u)-1$, so the number of edges colored $i$ incident to $u$ in $H$ is more than $(d_i(u)+1)/2-1$, hence at leat $d_i(u)/2$. This ends the proof.
\end{proof}
\section{Spanning connected 2-partitions}
\label{sec:spanning}
The following result is not difficult to prove (just take a max cut partition). For a generalization  to general $k$ partitions see \cite{bangJGT99}.

\begin{theorem}
  \label{thm:bipEC}
  Every graph $G=(V,E)$ of edge-connectivity $\lambda{}(G)$ has a spanning bipartite subgraph $H$ with $\lambda{}(H)\geq\lambda{}(G)/2$.
\end{theorem}

For digraphs no such result is possible due to the following.

\begin{theorem}\cite{bangJGT92}
  For every integer $k\geq 1$ there exists a $k$-strong eulerian digraph $D$ which has no spanning bipartite subdigraph in which every vertex has in- and out-degree at least one.
\end{theorem}

\begin{theorem}\cite{bangJGT92}
  It is NP-complete to decide for a given digraph $D$ whether $D$ has a spanning bipartite subdigraph in which every vertex has in- and out-degree at least one.
\end{theorem}

It is a well known result of Nash-Williams and Tutte that every $2k$-edge-connected graph has $k$ edge-disjoint spanning trees. Combining this with Theorem \ref{thm:bipEC} we get that every $4k$-edge-connected graph $G=(V,E)$ has a spanning bipartite subgraph $H$ with $k$ edge-disjoint spanning trees. The next result (Proposition \ref{prop:nogoodtrees}) shows that no similar result holds for monochromatic trees in 2-partitions of 2-edge-coloured graphs. To prove it we need the following easy observation.

\begin{lemma}
\label{lem1}
Let $B$ be any connected bipartite graph with partite sets $V_1$ and $V_2$.
Let $X$ and $Y$ be any partition of $V(B)$ and let $B(X,Y)$ denote the subgraph of $B$ induced by all edges between $X$ and $Y$.
Then $B(X,Y)$ is connected if and only if $\{X,Y\} = \{V_1,V_2\}$.
\end{lemma}

\pf 
First assume that $X \cap V_1 \not= \emptyset$ and $X \cap V_2 \not= \emptyset$.
Let $x_1 \in X \cap V_1$ and let $x_2 \in X \cap V_2$ be arbitrary.
Any path between $x_1$ and $x_2$ in $B(X,Y)$ would have an even number of edges (as $B(X,Y)$ is bipartite), 
but there is no $(x_1,x_2)$-path in $B$ with an even number of edges as $x_1 \in V_1$ and  $x_2 \in V_2$, 
a contradiction. So $x_1$ and $x_2$ cannot belong to the same connected component of $B(X,Y)$.
Hence if $B(X,Y)$ is connected then $X \cap V_1 = \emptyset$ or $X \cap V_2 = \emptyset$. 

Analogously, if $B(X,Y)$ is connected then $Y \cap V_1 = \emptyset$ or $Y \cap V_2 = \emptyset$, 
which implies that $\{X,Y\} = \{V_1,V_2\}$, as desired.~\qed

\begin{proposition}
  \label{prop:nogoodtrees}
  For every integer $k\geq 1$ there exists a 2-edge-colored graph $G=(V,E)$ such that each of the subgraphs
  $G_{blue}=(V,E_{blue})$ and $G_{red}=(V,E_{red})$ are $k$-edge-connected but $G$ has no spanning bipartite subgraph $H$ for which both the red and the blue subgraph of $H$ is connected.
\end{proposition}

\pf Take a $k$-edge connected $k$-regular bipartite graph $H'=(V_1\cup V_2,E_{blue})$ in which all edges are colored blue and $|V_1|=|V_2|\geq k+1$. Add a red perfect matching between the vertices of $V_1$ and $V_2$ such that no edge is both red and blue and finally add all possible red edges inside each $V_i$ and let $G$ be the resulting 2-edge-colored graph. It is easy to check that both the blue and red subgraph of $G$ is $k$-edge-connected. As $(V_1,V_2)$ is the only bipartition of $V$ for which the induced blue bipartite subgraph is connected the claim follows. \qed

%\input{pic_steps.tex}
%\item $E_1 = \{x_i^1 x_i^2, x_i^2 x_i^3, x_i^3 x_i^1 \; | \; i=1,2,\ldots,m \}$.
%\item $E_2 = \{s_i t_j \; | \; i,j \in \{1,2,3,4\} \}$.
%\item $E_3$ denotes all edges $x_i^j x_a^b$ where the $j$'th the literal in $C_i$ is the negation of the $b$'th literal in $C_a$.
%\item $E_4 = \{s_1 x_i^1, t_1 x_i^1 \; | \; i=1,2,\ldots,m \}$.
%\item For each $i=1,2,\ldots,n$ let $a_i$ and $b_i$ be any values such that the $b_i$'th literal in $C_{a_i}$ is the literal $v_i$ (there may be any possible options
%for $a_i$ and $b_i$ but we just pick an arbitrary one). Now let $E_5 = \{s_2 x_{a_i}^{b_i}, t_2 x_{a_i}^{b_i}  \; | \; i=1,2,\ldots,n \}$.

\begin{figure}[H]
  \begin{center}
\begin{tabular}{|c|c|c|c|c|} \hline
\tikzstyle{vertexL}=[circle,draw, minimum size=5pt, scale=0.2, inner sep=0.5pt]
\tikzstyle{vertexB}=[circle,draw, minimum size=5pt, scale=0.2, inner sep=0.5pt]
\tikzstyle{vertexR}=[circle,draw, color=red!100, minimum size=14pt, scale=0.6, inner sep=0.5pt]
\hspace{-0.6cm} \begin{tikzpicture}[scale=0.14]
  \node at (2,-3.3) {\mbox{ }};
  \node at (18,11) {\mbox{ }};
  \node at (1.5,3.5) {$E_1$};

  \node (x11) at (1,8) [vertexB]{$x_1^1$};
  \node (x12) at (3,8) [vertexB]{$x_1^2$};
  \node (x13) at (5,8) [vertexB]{$x_1^3$};

  \node (x21) at (8,8) [vertexB]{$x_2^1$};
  \node (x22) at (10,8) [vertexB]{$x_2^2$};
  \node (x23) at (12,8) [vertexB]{$x_2^3$};

  \node (x31) at (15,8) [vertexB]{$x_3^1$};
  \node (x32) at (17,8) [vertexB]{$x_3^2$};
  \node (x33) at (19,8) [vertexB]{$x_3^3$};

  \node (s4) at (3,1) [vertexL]{$s_4$};
  \node (t4) at (5,1) [vertexL]{$t_4$};
  \node (s1) at (7,1) [vertexL]{$s_1$};
  \node (t1) at (9,1) [vertexL]{$t_1$};
  \node (s2) at (11,1) [vertexL]{$s_2$};
  \node (t2) at (13,1) [vertexL]{$t_2$};
  \node (s3) at (15,1) [vertexL]{$s_3$};
  \node (t3) at (17,1) [vertexL]{$t_3$};

\withCol{{\color{red}
  \draw[line width=0.03cm] (x11) to (x12);
  \draw[line width=0.03cm] (x12) to (x13);
  \draw[line width=0.03cm] (x11) to [out=25, in=155] (x13);
  \draw[line width=0.03cm] (x21) to (x22);
  \draw[line width=0.03cm] (x22) to (x23);
  \draw[line width=0.03cm] (x21) to [out=25, in=155] (x23);
  \draw[line width=0.03cm] (x31) to (x32);
  \draw[line width=0.03cm] (x32) to (x33);
  \draw[line width=0.03cm] (x31) to [out=25, in=155] (x33);
}}

\noCol{
  \draw[line width=0.01cm] (x11) to (x12);
  \draw[line width=0.01cm] (x12) to (x13);
  \draw[line width=0.01cm] (x11) to [out=25, in=155] (x13);
  \draw[line width=0.01cm] (x21) to (x22);
  \draw[line width=0.01cm] (x22) to (x23);
  \draw[line width=0.01cm] (x21) to [out=25, in=155] (x23);
  \draw[line width=0.01cm] (x31) to (x32);
  \draw[line width=0.01cm] (x32) to (x33);
  \draw[line width=0.01cm] (x31) to [out=25, in=155] (x33);
}
   \end{tikzpicture} & 
\tikzstyle{vertexL}=[circle,draw, minimum size=5pt, scale=0.2, inner sep=0.5pt]
\tikzstyle{vertexB}=[circle,draw, minimum size=5pt, scale=0.2, inner sep=0.5pt]
\tikzstyle{vertexR}=[circle,draw, color=red!100, minimum size=14pt, scale=0.6, inner sep=0.5pt]
\hspace{-0.6cm} \begin{tikzpicture}[scale=0.14]
  \node at (2,-3.3) {\mbox{ }};
  \node at (18,11) {\mbox{ }};
  \node at (1.5,3.5) {$E_2$};

  \node (x11) at (1,8) [vertexB]{$x_1^1$};
  \node (x12) at (3,8) [vertexB]{$x_1^2$};
  \node (x13) at (5,8) [vertexB]{$x_1^3$};

  \node (x21) at (8,8) [vertexB]{$x_2^1$};
  \node (x22) at (10,8) [vertexB]{$x_2^2$};
  \node (x23) at (12,8) [vertexB]{$x_2^3$};

  \node (x31) at (15,8) [vertexB]{$x_3^1$};
  \node (x32) at (17,8) [vertexB]{$x_3^2$};
  \node (x33) at (19,8) [vertexB]{$x_3^3$};

  \node (s4) at (3,1) [vertexL]{$s_4$};
  \node (t4) at (5,1) [vertexL]{$t_4$};
  \node (s1) at (7,1) [vertexL]{$s_1$};
  \node (t1) at (9,1) [vertexL]{$t_1$};
  \node (s2) at (11,1) [vertexL]{$s_2$};
  \node (t2) at (13,1) [vertexL]{$t_2$};
  \node (s3) at (15,1) [vertexL]{$s_3$};
  \node (t3) at (17,1) [vertexL]{$t_3$};

\withCol{{\color{blue}
  \draw[line width=0.03cm] (s1) to (t1);
  \draw[line width=0.03cm] (t1) to (s2);
  \draw[line width=0.03cm] (s2) to (t2);
  \draw[line width=0.03cm] (t2) to (s3);
  \draw[line width=0.03cm] (s3) to (t3);
  \draw[line width=0.03cm] (t3) to [out=240, in=300] (s4);
  \draw[line width=0.03cm] (s4) to (t4);
  \draw[line width=0.03cm] (t4) to (s1);
}}

\noCol{
  \draw[line width=0.04cm] (s1) to (t1);
  \draw[line width=0.04cm] (t1) to (s2);
  \draw[line width=0.04cm] (s2) to (t2);
  \draw[line width=0.04cm] (t2) to (s3);
  \draw[line width=0.04cm] (s3) to (t3);
  \draw[line width=0.04cm] (t3) to [out=240, in=300] (s4);
  \draw[line width=0.04cm] (s4) to (t4);
  \draw[line width=0.04cm] (t4) to (s1);
}

\withCol{{\color{red}
  \draw[line width=0.03cm] (s4) to [out=330, in=210] (t1);
  \draw[line width=0.03cm] (t4) to [out=330, in=210] (s2);
  \draw[line width=0.03cm] (s1) to [out=330, in=210] (t2);
  \draw[line width=0.03cm] (t1) to [out=330, in=210] (s3);
  \draw[line width=0.03cm] (s2) to [out=330, in=210] (t3);
  \draw[line width=0.03cm] (s4) to [out=310, in=230] (t2);
  \draw[line width=0.03cm] (t4) to [out=310, in=230] (s3);
  \draw[line width=0.03cm] (s1) to [out=310, in=230] (t3);
}}

\noCol{
  \draw[line width=0.01cm] (s4) to [out=330, in=210] (t1);
  \draw[line width=0.01cm] (t4) to [out=330, in=210] (s2);
  \draw[line width=0.01cm] (s1) to [out=330, in=210] (t2);
  \draw[line width=0.01cm] (t1) to [out=330, in=210] (s3);
  \draw[line width=0.01cm] (s2) to [out=330, in=210] (t3);
  \draw[line width=0.01cm] (s4) to [out=310, in=230] (t2);
  \draw[line width=0.01cm] (t4) to [out=310, in=230] (s3);
  \draw[line width=0.01cm] (s1) to [out=310, in=230] (t3);
}
   \end{tikzpicture} & 
\tikzstyle{vertexL}=[circle,draw, minimum size=5pt, scale=0.2, inner sep=0.5pt]
\tikzstyle{vertexB}=[circle,draw, minimum size=5pt, scale=0.2, inner sep=0.5pt]
\tikzstyle{vertexR}=[circle,draw, color=red!100, minimum size=14pt, scale=0.6, inner sep=0.5pt]
\hspace{-0.6cm} \begin{tikzpicture}[scale=0.14]
  \node at (2,-3.3) {\mbox{ }};
  \node at (18,11) {\mbox{ }};
  \node at (1.5,3.5) {$E_3$};

  \node (x11) at (1,8) [vertexB]{$x_1^1$};
  \node (x12) at (3,8) [vertexB]{$x_1^2$};
  \node (x13) at (5,8) [vertexB]{$x_1^3$};

  \node (x21) at (8,8) [vertexB]{$x_2^1$};
  \node (x22) at (10,8) [vertexB]{$x_2^2$};
  \node (x23) at (12,8) [vertexB]{$x_2^3$};

  \node (x31) at (15,8) [vertexB]{$x_3^1$};
  \node (x32) at (17,8) [vertexB]{$x_3^2$};
  \node (x33) at (19,8) [vertexB]{$x_3^3$};

  \node (s4) at (3,1) [vertexL]{$s_4$};
  \node (t4) at (5,1) [vertexL]{$t_4$};
  \node (s1) at (7,1) [vertexL]{$s_1$};
  \node (t1) at (9,1) [vertexL]{$t_1$};
  \node (s2) at (11,1) [vertexL]{$s_2$};
  \node (t2) at (13,1) [vertexL]{$t_2$};
  \node (s3) at (15,1) [vertexL]{$s_3$};
  \node (t3) at (17,1) [vertexL]{$t_3$};

\withCol{{\color{blue}
  \draw[line width=0.03cm] (x11) to [out=30, in=150] (x21);
  \draw[line width=0.03cm] (x21) to [out=30, in=150] (x31);
  \draw[line width=0.03cm] (x12) to [out=30, in=150] (x22);
  \draw[line width=0.03cm] (x13) to [out=25, in=155] (x32);
  \draw[line width=0.03cm] (x23) to [out=30, in=150] (x33);
}}

\noCol{
  \draw[line width=0.04cm] (x11) to [out=30, in=150] (x21);
  \draw[line width=0.04cm] (x21) to [out=30, in=150] (x31);
  \draw[line width=0.04cm] (x12) to [out=30, in=150] (x22);
  \draw[line width=0.04cm] (x13) to [out=25, in=155] (x32);
  \draw[line width=0.04cm] (x23) to [out=30, in=150] (x33);
}
   \end{tikzpicture} & 
\tikzstyle{vertexL}=[circle,draw, minimum size=5pt, scale=0.2, inner sep=0.5pt]
\tikzstyle{vertexB}=[circle,draw, minimum size=5pt, scale=0.2, inner sep=0.5pt]
\tikzstyle{vertexR}=[circle,draw, color=red!100, minimum size=14pt, scale=0.6, inner sep=0.5pt]
\hspace{-0.6cm} \begin{tikzpicture}[scale=0.14]
  \node at (2,-3.3) {\mbox{ }};
  \node at (18,11) {\mbox{ }};
  \node at (1.5,3.5) {$E_4$};

  \node (x11) at (1,8) [vertexB]{$x_1^1$};
  \node (x12) at (3,8) [vertexB]{$x_1^2$};
  \node (x13) at (5,8) [vertexB]{$x_1^3$};

  \node (x21) at (8,8) [vertexB]{$x_2^1$};
  \node (x22) at (10,8) [vertexB]{$x_2^2$};
  \node (x23) at (12,8) [vertexB]{$x_2^3$};

  \node (x31) at (15,8) [vertexB]{$x_3^1$};
  \node (x32) at (17,8) [vertexB]{$x_3^2$};
  \node (x33) at (19,8) [vertexB]{$x_3^3$};

  \node (s4) at (3,1) [vertexL]{$s_4$};
  \node (t4) at (5,1) [vertexL]{$t_4$};
  \node (s1) at (7,1) [vertexL]{$s_1$};
  \node (t1) at (9,1) [vertexL]{$t_1$};
  \node (s2) at (11,1) [vertexL]{$s_2$};
  \node (t2) at (13,1) [vertexL]{$t_2$};
  \node (s3) at (15,1) [vertexL]{$s_3$};
  \node (t3) at (17,1) [vertexL]{$t_3$};

\withCol{{\color{red}
  \draw[line width=0.03cm] (s1) to (x11);
  \draw[line width=0.03cm] (t1) to (x11);
  \draw[line width=0.03cm] (s1) to (x21);
  \draw[line width=0.03cm] (t1) to (x21);
  \draw[line width=0.03cm] (s1) to (x31);
  \draw[line width=0.03cm] (t1) to (x31);
}}

\noCol{
  \draw[line width=0.01cm] (s1) to (x11);
  \draw[line width=0.01cm] (t1) to (x11);
  \draw[line width=0.01cm] (s1) to (x21);
  \draw[line width=0.01cm] (t1) to (x21);
  \draw[line width=0.01cm] (s1) to (x31);
  \draw[line width=0.01cm] (t1) to (x31);
}
   \end{tikzpicture} & 
\tikzstyle{vertexL}=[circle,draw, minimum size=5pt, scale=0.2, inner sep=0.5pt]
\tikzstyle{vertexB}=[circle,draw, minimum size=5pt, scale=0.2, inner sep=0.5pt]
\tikzstyle{vertexR}=[circle,draw, color=red!100, minimum size=14pt, scale=0.6, inner sep=0.5pt]
\hspace{-0.6cm} \begin{tikzpicture}[scale=0.14]
  \node at (2,-3.3) {\mbox{ }};
  \node at (18,11) {\mbox{ }};
  \node at (1.5,3.5) {$E_5$};

  \node (x11) at (1,8) [vertexB]{$x_1^1$};
  \node (x12) at (3,8) [vertexB]{$x_1^2$};
  \node (x13) at (5,8) [vertexB]{$x_1^3$};

  \node (x21) at (8,8) [vertexB]{$x_2^1$};
  \node (x22) at (10,8) [vertexB]{$x_2^2$};
  \node (x23) at (12,8) [vertexB]{$x_2^3$};

  \node (x31) at (15,8) [vertexB]{$x_3^1$};
  \node (x32) at (17,8) [vertexB]{$x_3^2$};
  \node (x33) at (19,8) [vertexB]{$x_3^3$};

  \node (s4) at (3,1) [vertexL]{$s_4$};
  \node (t4) at (5,1) [vertexL]{$t_4$};
  \node (s1) at (7,1) [vertexL]{$s_1$};
  \node (t1) at (9,1) [vertexL]{$t_1$};
  \node (s2) at (11,1) [vertexL]{$s_2$};
  \node (t2) at (13,1) [vertexL]{$t_2$};
  \node (s3) at (15,1) [vertexL]{$s_3$};
  \node (t3) at (17,1) [vertexL]{$t_3$};

\withCol{{\color{blue}
  \draw[line width=0.03cm] (s2) to (x11);
  \draw[line width=0.03cm] (t2) to (x11);
  \draw[line width=0.03cm] (s2) to (x12);
  \draw[line width=0.03cm] (t2) to (x12);
  \draw[line width=0.03cm] (s2) to (x13);
  \draw[line width=0.03cm] (t2) to (x13);
  \draw[line width=0.03cm] (s2) to (x23);
  \draw[line width=0.03cm] (t2) to (x23);
}}

\noCol{
  \draw[line width=0.04cm] (s2) to (x11);
  \draw[line width=0.04cm] (t2) to (x11);
  \draw[line width=0.04cm] (s2) to (x12);
  \draw[line width=0.04cm] (t2) to (x12);
  \draw[line width=0.04cm] (s2) to (x13);
  \draw[line width=0.04cm] (t2) to (x13);
  \draw[line width=0.04cm] (s2) to (x23);
  \draw[line width=0.04cm] (t2) to (x23);
}
   \end{tikzpicture} \\ \hline %xxxxxxxxxxxxx
\multicolumn{5}{|c|}{
\tikzstyle{vertexL}=[circle,draw, minimum size=20pt, scale=1.0, inner sep=0.5pt]
\tikzstyle{vertexB}=[circle,draw, minimum size=20pt, scale=1.0, inner sep=0.5pt]
\tikzstyle{vertexR}=[circle,draw, color=red!100, minimum size=14pt, scale=0.6, inner sep=0.5pt]
\begin{tikzpicture}[scale=0.65]
  \node at (1,3.5) {{\huge $G$}};

  \node (x11) at (1,8) [vertexB]{$x_1^1$};
  \node (x12) at (3,8) [vertexB]{$x_1^2$};
  \node (x13) at (5,8) [vertexB]{$x_1^3$};

  \node (x21) at (8,8) [vertexB]{$x_2^1$};
  \node (x22) at (10,8) [vertexB]{$x_2^2$};
  \node (x23) at (12,8) [vertexB]{$x_2^3$};

  \node (x31) at (15,8) [vertexB]{$x_3^1$};
  \node (x32) at (17,8) [vertexB]{$x_3^2$};
  \node (x33) at (19,8) [vertexB]{$x_3^3$};

  \node (s4) at (3,1) [vertexL]{$s_4$};
  \node (t4) at (5,1) [vertexL]{$t_4$};
  \node (s1) at (7,1) [vertexL]{$s_1$};
  \node (t1) at (9,1) [vertexL]{$t_1$};
  \node (s2) at (11,1) [vertexL]{$s_2$};
  \node (t2) at (13,1) [vertexL]{$t_2$};
  \node (s3) at (15,1) [vertexL]{$s_3$};
  \node (t3) at (17,1) [vertexL]{$t_3$};

\withCol{{\color{blue}
  \draw[line width=0.04cm] (s2) to (x11);
  \draw[line width=0.04cm] (t2) to (x11);
  \draw[line width=0.04cm] (s2) to (x12);
  \draw[line width=0.04cm] (t2) to (x12);
  \draw[line width=0.04cm] (s2) to (x13);
  \draw[line width=0.04cm] (t2) to (x13);
  \draw[line width=0.04cm] (s2) to (x23);
  \draw[line width=0.04cm] (t2) to (x23);

  \draw[line width=0.04cm] (x11) to [out=30, in=150] (x21);
  \draw[line width=0.04cm] (x21) to [out=30, in=150] (x31);

  \draw[line width=0.04cm] (x12) to [out=30, in=150] (x22);
  \draw[line width=0.04cm] (x13) to [out=25, in=155] (x32);
  \draw[line width=0.04cm] (x23) to [out=30, in=150] (x33);

  \draw[line width=0.04cm] (s1) to (t1);
  \draw[line width=0.04cm] (t1) to (s2);
  \draw[line width=0.04cm] (s2) to (t2);
  \draw[line width=0.04cm] (t2) to (s3);
  \draw[line width=0.04cm] (s3) to (t3);
  \draw[line width=0.04cm] (t3) to [out=240, in=300] (s4);
  \draw[line width=0.04cm] (s4) to (t4);
  \draw[line width=0.04cm] (t4) to (s1);
}}

\noCol{
  \draw[line width=0.07cm] (s2) to (x11);
  \draw[line width=0.07cm] (t2) to (x11);
  \draw[line width=0.07cm] (s2) to (x12);
  \draw[line width=0.07cm] (t2) to (x12);
  \draw[line width=0.07cm] (s2) to (x13);
  \draw[line width=0.07cm] (t2) to (x13);
  \draw[line width=0.07cm] (s2) to (x23);
  \draw[line width=0.07cm] (t2) to (x23);

  \draw[line width=0.07cm] (x11) to [out=30, in=150] (x21);
  \draw[line width=0.07cm] (x21) to [out=30, in=150] (x31);

  \draw[line width=0.07cm] (x12) to [out=30, in=150] (x22);
  \draw[line width=0.07cm] (x13) to [out=25, in=155] (x32);
  \draw[line width=0.07cm] (x23) to [out=30, in=150] (x33);

  \draw[line width=0.07cm] (s1) to (t1);
  \draw[line width=0.07cm] (t1) to (s2);
  \draw[line width=0.07cm] (s2) to (t2);
  \draw[line width=0.07cm] (t2) to (s3);
  \draw[line width=0.07cm] (s3) to (t3);
  \draw[line width=0.07cm] (t3) to [out=240, in=300] (s4);
  \draw[line width=0.07cm] (s4) to (t4);
  \draw[line width=0.07cm] (t4) to (s1);
}

\withCol{{\color{red}
  \draw[line width=0.04cm] (x11) to (x12);
  \draw[line width=0.04cm] (x12) to (x13);
  \draw[line width=0.04cm] (x11) to [out=25, in=155] (x13);
  \draw[line width=0.04cm] (x21) to (x22);
  \draw[line width=0.04cm] (x22) to (x23);
  \draw[line width=0.04cm] (x21) to [out=25, in=155] (x23);
  \draw[line width=0.04cm] (x31) to (x32);
  \draw[line width=0.04cm] (x32) to (x33);
  \draw[line width=0.04cm] (x31) to [out=25, in=155] (x33);

  \draw[line width=0.04cm] (s1) to (x11);
  \draw[line width=0.04cm] (t1) to (x11);
  \draw[line width=0.04cm] (s1) to (x21);
  \draw[line width=0.04cm] (t1) to (x21);
  \draw[line width=0.04cm] (s1) to (x31);
  \draw[line width=0.04cm] (t1) to (x31);

  \draw[line width=0.04cm] (s4) to [out=330, in=210] (t1);
  \draw[line width=0.04cm] (t4) to [out=330, in=210] (s2);
  \draw[line width=0.04cm] (s1) to [out=330, in=210] (t2);
  \draw[line width=0.04cm] (t1) to [out=330, in=210] (s3);
  \draw[line width=0.04cm] (s2) to [out=330, in=210] (t3);
  \draw[line width=0.04cm] (s4) to [out=310, in=230] (t2);
  \draw[line width=0.04cm] (t4) to [out=310, in=230] (s3);
  \draw[line width=0.04cm] (s1) to [out=310, in=230] (t3);
}}

\noCol{
  \draw[line width=0.02cm] (x11) to (x12);
  \draw[line width=0.02cm] (x12) to (x13);
  \draw[line width=0.02cm] (x11) to [out=25, in=155] (x13);
  \draw[line width=0.02cm] (x21) to (x22);
  \draw[line width=0.02cm] (x22) to (x23);
  \draw[line width=0.02cm] (x21) to [out=25, in=155] (x23);
  \draw[line width=0.02cm] (x31) to (x32);
  \draw[line width=0.02cm] (x32) to (x33);
  \draw[line width=0.02cm] (x31) to [out=25, in=155] (x33);

  \draw[line width=0.02cm] (s1) to (x11);
  \draw[line width=0.02cm] (t1) to (x11);
  \draw[line width=0.02cm] (s1) to (x21);
  \draw[line width=0.02cm] (t1) to (x21);
  \draw[line width=0.02cm] (s1) to (x31);
  \draw[line width=0.02cm] (t1) to (x31);

  \draw[line width=0.02cm] (s4) to [out=330, in=210] (t1);
  \draw[line width=0.02cm] (t4) to [out=330, in=210] (s2);
  \draw[line width=0.02cm] (s1) to [out=330, in=210] (t2);
  \draw[line width=0.02cm] (t1) to [out=330, in=210] (s3);
  \draw[line width=0.02cm] (s2) to [out=330, in=210] (t3);
  \draw[line width=0.02cm] (s4) to [out=310, in=230] (t2);
  \draw[line width=0.02cm] (t4) to [out=310, in=230] (s3);
  \draw[line width=0.02cm] (s1) to [out=310, in=230] (t3);
}

%  \node at (1,5) {$T_1$};
   \end{tikzpicture} } \\ \hline
\end{tabular}
\caption{The graph $G$ (and an illustration of the edge-sets $E_1,E_2,E_3,E_4,E_5$) 
in the proof of Theorem~\ref{NPhard}. The clauses of $I$ are $C_1=(v_1,v_2,v_3)$, $C_2=(\overline{v_1},\overline{v_2},v_4)$
and $C_3=(v_1,\overline{v_3},\overline{v_4})$. \noCol{The thick edge are "blue" and the thin edges are "red".}\JBJ{In $E_5$ we have taken $a_1=a_2=a_3=1$ and $a_4=2$.}}\label{fig:no2colpart}
\end{center}
\end{figure}
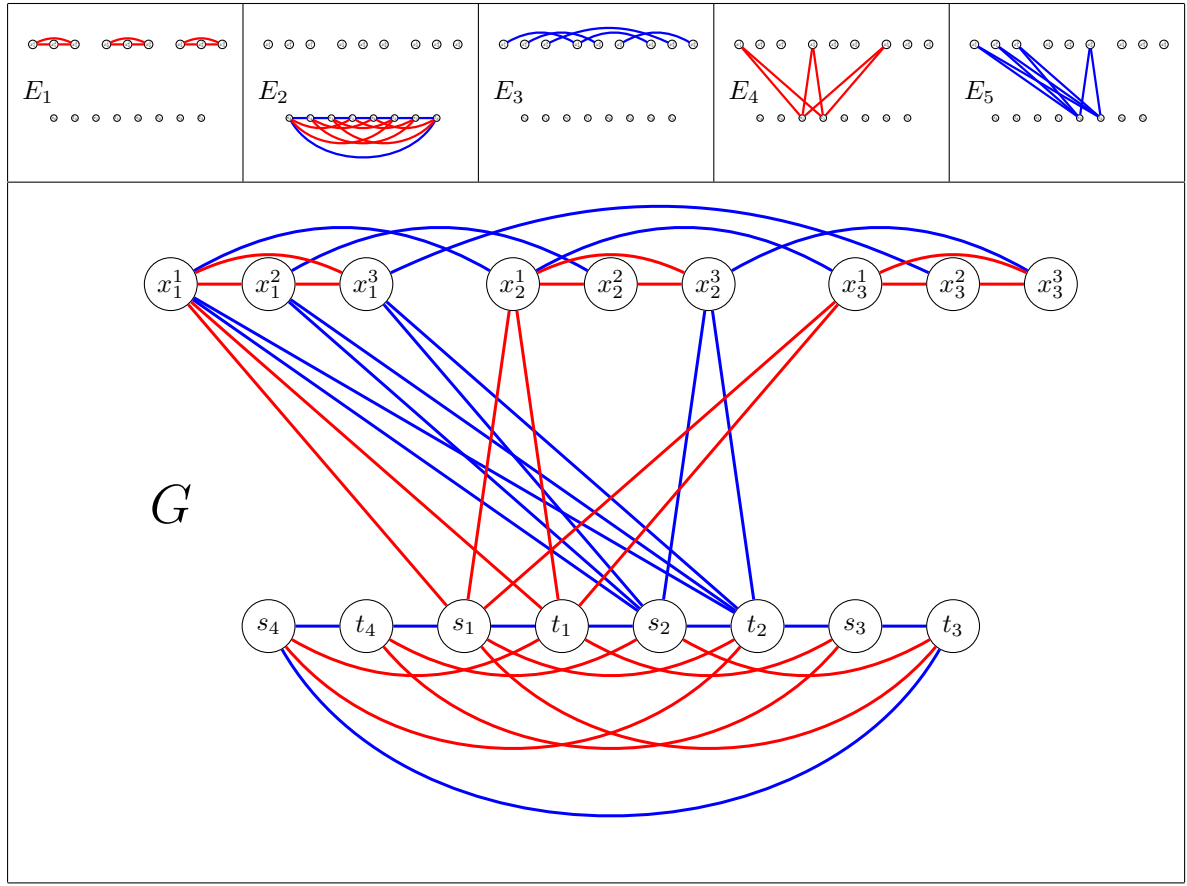

\begin{theorem}
\label{NPhard}
Let $G$ be a $2$-edge-colored graph. Deciding if there exists a spanning bipartite subgraph, $B$, of $G$, such that the edges of $B$ of each of the colors induce a spanning connected subgraph of $G$ is NP-complete.
\end{theorem}

\pf We will reduce from {\em not-all-equal-$3$-SAT}, which is denoted by NAE-$3$-SAT for short. 
Let $I$ be an instance of NAE-$3$-SAT, with variables $v_1,v_2,\ldots,v_n$ and clauses $C_1,C_2,\ldots,C_m$.   
That is, each $C_i$ contain three literals (which is either a variable $v_j$ or the negation of a variable $\overline{v_j}$) and the purpose
is to find a truth assignment to all variables such that each clause has both positive and negative literals.
We may assume that no variable and its negation appear in the same clause, as we may simply remove such a clause without changing the problem.
We may also assume that each variable, $v_j$, appears as a literal in some clause and the negation, $\overline{v_j}$, of each variable also appears as a literal 
in some clause (otherwise we can put in 2 extra clauses $(v_j,a,b)$ and $(\overline{v_j},\overline{a},\overline{b})$ without changing the problem, 
where $a$ and $b$ are new variables). We now construct a $2$-edge-colored graph $G$ as follows.

\begin{itemize}
\item Let $V(G)=\{x_i^j \; | \; i=1,2,\ldots,m \mbox{ and } j=1,2,3 \} \cup \{s_1,s_2,s_3,s_4,t_1,t_2,t_3,t_4\}$
\item Let $E(G)= E_1 \cup E_2 \cup E_3 \cup E_4 \cup E_5$, where 
\begin{itemize}
\item $E_1 = \{x_i^1 x_i^2, x_i^2 x_i^3, x_i^3 x_i^1 \; | \; i=1,2,\ldots,m \}$.
\item $E_2 = \{s_i t_j \; | \; i,j \in \{1,2,3,4\} \}$. 
\item $E_3$ denotes all edges $x_i^j x_a^b$ where the $j$'th the literal in $C_i$ is the negation of the $b$'th literal in $C_a$. 
\item $E_4 = \{s_1 x_i^1, t_1 x_i^1 \; | \; i=1,2,\ldots,m \}$.
\item For each $i=1,2,\ldots,n$ let $a_i$ and $b_i$ be any values such that the $b_i$'th literal in $C_{a_i}$ is the literal $v_i$ (there may be many possible options
for $a_i$ and $b_i$ but we just pick an arbitrary one). Now let $E_5 = \{s_2 x_{a_i}^{b_i}, t_2 x_{a_i}^{b_i}  \; | \; i=1,2,\ldots,n \}$.
\end{itemize}
\end{itemize}

We note that the edges of $E_2$ induce a $K_{4,4}$, which contains two edge disjoint Hamilton cycles. Let the edges of one of the Hamilton cycles be colored red and 
the edges of the other Hamilton cycle be colored blue.
Let the color of all edges in $E_1 \cup E_4$ be red and let the color of all edges in $E_3 \cup E_5$ be blue.
This completes the definition of the $2$-edge-colored graph $G$. See Figure \ref{fig:no2colpart}. It is easy to check  that the red edges of $E_2$ form a  Hamilton cycle.

We will now show that $I$ is not-all-equal-satisfiable if and only if there exists a spanning bipartite subgraph, $B$ of $G$, where the red edges of $B$ induce a connected
spanning subgraph of $G$ and the blue edges  of $B$ induce a connected spanning subgraph of $G$.

First assume that $I$ is not-all-equal-satisfiable and that we are given the truth values of the variables $v_1,v_2, \ldots, v_n$ such that all clauses are
not-all-equal-satisfied. We initially let $X=\{s_i,s_2,s_3,s_4\}$ and $Y=\{t_1,t_2,t_3,t_4\}$. 
Whenever $v_k$ is true  we put $x_i^j$ in $X$ whenever the $j$'th literal in $C_i$ is $v_k$ and we put 
$x_i^j$ in $Y$ whenever the $j$'th literal in $C_i$ is $\overline{v_k}$. Analogously, whenever $v_k$ is false  we put $x_i^j$ in $Y$ whenever the $j$'th 
literal in $C_i$ is $v_k$ and we put $x_i^j$ in $X$ whenever the $j$'th literal in $C_i$ is $\overline{v_k}$. 
This partitions $V(G)$ into $X$ and $Y$. Let $B$ be the bipartite subgraph of $G$ induced by the partite sets $X$ and $Y$. Note that, by construction, $X$ ($Y$) contains all vertices of $G$ corresponding to literals that are true (false) under the given truth assignment.

Let $B_{blue}$ denote the subgraph containing all edges of $B$ colored blue and let $B_{red}$ denote the subgraph containing all edges of $B$ colored red. 
By the definition of the colors of $E_2$ (note that all edges of $E_2$ belong to $B$) we note that $s_1,s_2,s_3,s_4,t_1,t_2,t_3,t_4$ all belong to the same
connected component, $C_{red}$, of $B_{red}$ and to the same connected component, $C_{blue}$, of $B_{blue}$. 
By the edges in $E_4$ we note that $x_i^1$ also belongs to $C_{red}$ for all $i=1,2,\ldots,m$ (as either $s_1 x_i^1$ belongs to $B_{red}$
or $t_1 x_i^1$ belongs to $B_{red}$). Due to the edges in $E_1$ this implies that $V(C_{red})=V(G)$ as $I$ was not-all-equal-satisfiable, which implies
that two of the edges in $\{x_i^1 x_i^2, x_i^2 x_i^3, x_i^3 x_i^1\}$ belong to $B_{red}$.

By the edges in $E_5$ we note that $x_{a_i}^{b_i}$ belongs to $C_{blue}$ for all $i=1,2,\ldots,n$ (as either $s_2 x_{a_i}^{b_i}$ belongs to $B_{blue}$
or $t_2 x_{a_i}^{b_i}$ belongs to $B_{blue}$). Due to the edges in $E_3$, which all belong to $B_{blue}$ this implies that $V(C_{blue})=V(G)$ 
(as both literal $v_k$ and $\overline{v_k}$ occur in $I$ for all $k=1,2,\ldots,n$). Therefore $B_{red}$ and $B_{blue}$ are both spanning connected subgraphs of $G$,
as desired.

\2

Conversely assume that there exists a spanning bipartite subgraph, $B$ of $G$, where the red edges of $B$ induce a connected
spanning subgraph of $G$ and the blue edges  of $B$ induce a connected spanning subgraph of $G$. We will now show that $I$ is not-all-equal-satisfiable.
Let $X$ and $Y$ be the partite sets of $B$ and let $B_{red}$ be the subgraph of $B$ induced by all the red edges of $B$. We will first show the following
two claims.

\2

{\bf Claim A:} {\em $\{x_i^1,x_i^2,x_i^3\} \cap X \not= \emptyset$ and $\{x_i^1,x_i^2,x_i^3\} \cap Y \not= \emptyset$ for all $i=1,2,\ldots,m$.}

{\bf Proof of Claim A:}
If $\{x_i^1,x_i^2,x_i^3\} \subseteq X$ then $B_{red}$ is not connected as $x_i^2$ and $x_i^3$ won't be incident with any red edges in $B_{red}$. 
Analogously, if $\{x_i^1,x_i^2,x_i^3\} \subseteq Y$ then $B_{red}$ is not connected as again $x_i^2$ and $x_i^3$ won't be incident with any red edges in $B_{red}$. 
So we must have $\{x_i^1,x_i^2,x_i^3\} \cap X \not= \emptyset$ and $\{x_i^1,x_i^2,x_i^3\} \cap Y \not= \emptyset$, as desired.

\2

{\bf Claim B:} {\em If the $j$'th literal in $C_i$ is $v_k$ and the $p$'th literal in $C_q$ is $\overline{v_k}$ then $x_i^j x_q^p \in E(B)$.}

{\bf Proof of Claim B:} Recall that $E_3$ denotes all edges $x_i^j x_a^b$ where the $j$'th the literal in $C_i$ is the negation of the $b$'th literal in $C_a$.
Let $E_3^k$ be all edges of $E_3$ where the above literal is $v_k$ or $\overline{v_k}$. We note that $E_3^k$ induces a bipartite graph, $B_3^k$, where
the partite sets correspond to the vertices of $G$ that correspond to $v_k$ and $\overline{v_k}$, respectively.
So $B_3^k$ is connected and bipartite. As only one vertex of $B_3^k$ has blue edges to vertices not in $V(B_3^k)$ in $G$, Lemma~\ref{lem1} implies
that all edges of $B_3^k$ must belong to $B$, which completes the proof of Claim~B.

\2

By Claim~B we can let $v_k$ be true if $x_i^j \in X$ for some $i$ and $j$ where the $j$'th clause in $C_i$ is $v_k$ and we can let
$v_k$ be false if $x_i^j \in Y$ for some $i$ and $j$ where the $j$'th clause in $C_i$ is $v_k$. By Claim~B the above definition is well-defined.
Note that if $x_i^j \in X$ then the $j$'th literal in $C_i$ is true and if $x_i^j \in Y$ then the $j$'th literal in $C_i$ is false.
By Claim~A we note that the above truth assignment implies that $I$ is not-all-equal-satisfiable, which completes the proof.~\qed

\section{Remarks and open problems}\label{sec:remarks}

The graph $G$ which we constructed in Section \ref{sec:no3part} has 147 vertices and there may exist smaller 2-edge-colored graphs with no majority 3-partition. 

  \begin{problem}
   Determine the largest $N$ such that every 2-edge-colored graph $H$ on $N$ vertices has a majority 3-partition.
    \end{problem}

Each of the following questions deal with questions concerning possible   majority 3-partitions.

\begin{itemize}
\item Does every 2-edge-colored complete graph have a majority 3-partition?
\item Does every eulerian  (red degree equals blue degree at every vertex)  2-edge-colored graph have a majority 3-partition?
\item Does there exist an integer $k$ such that every 2-edge-colored graph $G$ in which every vertex $v$ has red and blue degree at least $k$ has a majority 3-partition?
  \item Let $G=(V,E)$ be a 2-edge-colored graph and call a subset $X\subset V$ {\bf good} if each vertex $v\in X$ has at least half of its blue neighbours outside $X$ and at least half of its red neighbours outside $X$. Question: Does every 2-edge-colored graph $G$ on $n$ vertices have a good subset of size at least $n/3$? Clearly if $G$ has a majority 3-partition, then it has several such sets.
\end{itemize}


\begin{thebibliography}{10}

\bibitem{anastosEJC28}
M.~Anastos, A.~Lamaison, R.~Steiner, and T.~Szab{\'{o}}.
\newblock Majority colorings of sparse digraphs.
\newblock {\em Electron. J. Comb.}, 28(2):2, 2021.

\bibitem{anholcerDM348}
M.~Anholcer, B.~Bosek, J.~Grytczuk, G.~Gutowski, J.~Przybylo, and M.~Zajac.
\newblock Mrs. correct and majority colorings.
\newblock {\em Discrete Mathematics}, 348:114577, 2025.

\bibitem{bangTCS719}
J.~Bang{-}Jensen, S.~Bessy, F.~Havet, and A.~Yeo.
\newblock Out-degree reducing partitions of digraphs.
\newblock {\em Theor. Comput. Sci.}, 719:64--72, 2018.

\bibitem{bangJGT92}
J.~Bang{-}Jensen, S.~Bessy, F.~Havet, and A.~Yeo.
\newblock Bipartite spanning sub(di)graphs induced by 2-partitions.
\newblock {\em J. Graph Theory}, 92(2):130--151, 2019.

\bibitem{bang2009}
J.~Bang-Jensen and G.~Gutin.
\newblock {\em {Digraphs: Theory, Algorithms and Applications}}.
\newblock Springer-Verlag, London, 2nd edition, 2009.

\bibitem{bangJGT99}
J.~Bang{-}Jensen, F.~Havet, M.~Kriesell, and A.~Yeo.
\newblock Low chromatic spanning sub(di)graphs with prescribed degree or
  connectivity properties.
\newblock {\em J. Graph Theory}, 99(4):615--636, 2022.

\bibitem{erdosIJM3}
P.~Erd\"os.
\newblock On some extremal problems in graph theory.
\newblock {\em Israel J. Math.}, 3:113--116, 1965.

\bibitem{giraoCPC26}
A.~Gir{\~{a}}o, T.~Kittipassorn, and K.~Popielarz.
\newblock Generalized majority colourings of digraphs.
\newblock {\em Comb. Probab. Comput.}, 26(6):850--855, 2017.

\bibitem{jiICMAI}
Q.~Ji and J.~Wang.
\newblock {Majority Colorings of Some Special Digraphs}.
\newblock In {\em Proceedings of the 8th International Conference on
  Mathematics and Artificial Intelligence, {ICMAI} 2023, Chongqing, China,
  April 7-9, 2023}, pages 31--34. {ACM}, 2023.

\bibitem{kreutzerEJC24}
S.~Kreutzer, S.~Oum, P.D. Seymour, D.~van~der Zypen, and D.R Wood.
\newblock Majority colourings of digraphs.
\newblock {\em Electron. J. Comb.}, 24(2):2, 2017.

\bibitem{khotFOCS43}
Khot. S.
\newblock {Hardness Results for Coloring 3 -Colorable 3 -Uniform Hypergraphs}.
\newblock In {\em 43rd Symposium on Foundations of Computer Science {(FOCS}
  2002), 16-19 November 2002, Vancouver, BC, Canada, Proceedings}, pages
  23--32. {IEEE} Computer Society, 2002.

\end{thebibliography}
\end{document}